\documentclass[a4paper,10pt]{article}

\usepackage{geometry}
\usepackage{epsfig}
\usepackage{amsfonts}
\usepackage{amsmath}
\usepackage{amssymb}
\usepackage{amsthm}
\usepackage{mathrsfs}
\usepackage{color}
\usepackage{amscd}
\usepackage{euscript}
\usepackage{tikz-cd}
\usepackage{stmaryrd}
\usepackage{mathtools}
\usepackage[utf8]{inputenc}
\usepackage{fancyhdr}
\usepackage[english]{babel}
\DeclareMathOperator{\spn}{span}
\usepackage{graphicx}
\graphicspath{ {./scrivania/} }
\usepackage{afterpage}
\usepackage{authblk}
\usepackage{enumitem}

\theoremstyle{plain} 
\newtheorem{thm}{Theorem}[section] 
\newtheorem{cor}[thm]{Corollary} 
\newtheorem{lem}[thm]{Lemma} 
\newtheorem{prop}[thm]{Proposition} 
\theoremstyle{definition} 
\newtheorem{defn}{Definition}[section] 
\newtheorem{oss}{Remark}

\newtheorem{ex}{Example}

\newtheorem*{imposs}{Important remark}

\DeclareFontEncoding{LS1}{}{}
\DeclareFontSubstitution{LS1}{stix}{m}{n}
\DeclareSymbolFont{symbols2}{LS1}{stixfrak} {m} {n}
\DeclareMathSymbol{\operp}{\mathbin}{symbols2}{"A8}
\usepackage[backend=bibtex,style=numeric]{biblatex}
\usepackage{graphicx} 

\title{Polyharmonicity, Almansi-type decompositions and Fueter-Sce theorem for several Clifford variables}

\author{Giulio Binosi\footnote{ORCID: \texttt{0000-0002-4733-6180}} \\
\small Dipartimento di Matematica ed Informatica "U. Dini", Universit\`a di Firenze\\ 
\small Viale Morgagni 67/A, I-50134 Firenze, Italy\\
\small binosi@altamatematica.it
}
\date{{\small
\textit{
}}}

\begin{document}

\maketitle

\begin{abstract}
    We study some harmonic properties of slice regular functions in one and several Clifford variables and give explicit formulas of the iterated Laplacian applied to slice regular functions and to their spherical derivative, which are new also in the one variable context. We propose several Almansi-type decompositions for slice functions in several Clifford variables. As a consequence, we establish a several variables version of Fueter-Sce theorem.
\end{abstract}

\section{Introduction}
Complex analysis is one of the most fundamental and successful branches of mathematics. The theory of holomorphic functions, which are solutions of the Wirtinger operator $\partial/\partial\overline{z} = \frac{1}{2}(\partial_\alpha + i\partial_\beta)$, is well-developed and remarkably rich. Following the discovery of hypercomplex algebras, such as the quaternions $\mathbb{H}$ and Clifford algebras $\mathbb{R}_m$, considerable efforts have been made to extend complex analysis to these settings.

The first significant breakthrough in this direction was achieved by R. Fueter in 1934 \cite{Fueter}, who introduced a notion of regularity in the non-commutative algebra of quaternions. Fueter's approach was based on generalizing the Wirtinger operator to $\partial/\partial\overline{q} = \frac{1}{2}(\partial_\alpha + i\partial_\beta + j\partial_\gamma + k\partial_\delta$) and defining as \emph{regular} (now known as Fueter-regular) those functions that lie in its kernel. This concept was subsequently extended to Clifford algebras $\mathbb{R}_m$ using the Dirac operator $\overline{\partial} = \frac{1}{2}(\partial_{x_0} + \sum_{i=1}^m e_i \partial_{x_i})$. The corresponding class of regular functions, known as monogenic functions, forms the foundation of Clifford analysis \cite{CliffordAnalysis}.

While Fueter-regular and monogenic functions successfully replicate several key aspects of holomorphic function theory and provide a refinement of harmonic analysis in Euclidean spaces, they suffer from significant algebraic limitations. Notably, basic functions such as the identity function, polynomials, and power series of the form $\sum_m x^n a_n$ (with $a_n \in \mathbb{R}_m$) fail to be Fueter-regular or monogenic. Moreover, these classes of functions are not closed under multiplication or composition, which makes it even more challenging to construct examples of monogenic functions.

To address these shortcomings, Gentili and Struppa introduced a novel approach in 2006, known as slice analysis \cite{GentiliStruppa}. Drawing on an idea originally proposed by Cullen \cite{Cullen}, they leveraged the complex-slice structure of $\mathbb{H}$ to define a new class of functions (called slice regular functions) by requiring their restrictions to each slice to be holomorphic. This framework has since gained substantial interest and undergone rapid development. Slice regularity was soon extended to Clifford algebras \cite{Cliffordsetting} and, more generally, to real alternative $^\ast$-algebras \cite{SRFonAA}. The latter work, by Ghiloni and Perotti, introduced the broader concept of slice functions, which need not satisfy any regularity condition. These functions are constructed from complex-intrinsic functions, called stem functions. The stem functions approach allowed the two authors to construct a several variable version of the theory in any alternative real $^\ast$-algebra \cite{Several}. The theory of one-variable slice regular functions is now well established, as summarized in the monograph \cite{LibroCaterina}.

Despite their apparent incompatibility (for example only constant functions belong to both classes) slice regular and monogenic functions are closely connected through Fueter’s theorem. Originally devised to construct monogenic functions, Fueter’s theorem asserts that applying the Laplacian $\Delta_4$ to a slice regular quaternionic function $f$ yields a Fueter-regular function. This result was later generalized by Sce \cite{sce} to Clifford algebras $\mathbb{R}_m$ with an even number of imaginary units, using $\Delta_{m+1}^{(m-1)/2}$ (where $\gamma_m = (m-1)/2$ is the so-called Sce exponent \cite{MicheleSceworks}), and was finally extended to odd-dimensional cases by Qian \cite{qian}. Consequently, the Fueter-Sce theorem serves as a bridge between slice and monogenic function theories, with the Fueter-Sce mapping $\Delta_{m+1}^{\gamma_m}$ as its core structure.
Another link between these theories arises from the relationship between the spherical derivative and the Dirac operator applied to slice regular functions: they coincide up to the multiplicative factor $\gamma_m$. This leads to an alternative formulation of the Fueter-Sce theorem, stating that the spherical derivative $f'_s$ of an $\mathbb{R}_m$-valued slice regular function $f$ is $\gamma_m$-polyharmonic, namely $\Delta_{m+1}^{\gamma_m}f'_s=0$. More generally, slice regular functions themselves satisfy $\overline{\partial} \Delta_{m+1}^{\gamma_m} f = 0$,hence, they are in particular $(\gamma+1)$-polyharmonic.

In 1899, Emilio Almansi's study of polyharmonic functions lead to the famous Almansi theorem \cite{AlmansiClassico}, which expresses any polyharmonic function $f$ of degree $p$ as a sum of harmonic functions weighted by powers of $|x|^2$. Similar decompositions have been explored in various contexts, including Dunkl analysis \cite{AlmansiDunkl}, discrete umbral calculus \cite{discreteAlmansi}, and hypercomplex analysis \cite{Harmonicity, Bisiharmonicity, AlmansiPolymonogenic}.

The goal of this paper is to further investigate the harmonic properties of slice regular functions, particularly in the setting of several Clifford variables. We derive explicit formulas for the iterative application of the Laplacian to both spherical derivatives and slice regular functions, revealing their polyharmonic nature when $m$ is odd. Additionally, we establish Almansi-type decompositions for slice regular functions in several Clifford variables. Finally, we extend the Fueter-Sce theorem to the case of several variables.

The paper is structured as follows. In Section 2, we provide the necessary preliminaries on Clifford algebras, monogenic functions, and slice regular functions in both one and several variables. Section 3 focuses on the harmonic properties of slice regular functions, deriving formulas for iterated Laplacians acting on spherical derivatives (Proposition \ref{Prop potenza laplaciano derivata sferica}) and on slice regular functions themselves (Theorem \ref{theorem laplacian any order slice regular functions}). These results establish their degrees of harmonicity as $\gamma_m$ and $\gamma_m+1$, respectively. We then extend these results to several variables (Proposition \ref{Prop:iteratedLaplaciansevvar} and Theorem \ref{thm:iteratedlaplaciansevvar}).

In Section 4, we present Almansi-type decompositions for slice functions in several Clifford variables (Theorem \ref{Teorema principale}). The complexity of the higher-dimensional setting results in $2^n$ different decompositions, each determined by a choice of variables. The components are given explicitely through partial spherical derivatives. In particular, they are circular with respect to the chosen variables that determine the decomposition; if, moreover, the decomposing function is slice regular, they are $\gamma_m$-polyharmonic in the same variables, too. 
We also prove the unique character of these decompositions and a new one-variable characterization of slice regularity, through the components of ordered Almansi-type decompositions (Proposition \ref{prop:3claims}). Combining these results with Almansi’s theorem, we obtain a decomposition into harmonic, spherical components (Corollary \ref{cor further decomposition several variables}).

Finally, in Section 5, we formulate and prove the Fueter-Sce theorem in several Clifford variables (Theorem \ref{thm fueter sce several variables}), presenting two distinct proofs: one leveraging the harmonic properties established in Section 3 and the other relying on Almansi-type decompositions.

\section{Preliminaries}
\subsection{Clifford algebras and monogenic functions}
Let $m\in\mathbb{N}$, let $\{e_0,e_1,...,e_m\}$ be an orthonormal basis of $\mathbb{R}^{m+1}$ and let us define the following product rule
\begin{equation}
\label{eq product clifford}
    \begin{split}
&e_i\cdot e_0=e_0\cdot e_i=e_i,\qquad\forall i=1,...,m\\
&e_i\cdot e_j+e_j\cdot e_i=-2\delta_{ij},\qquad\forall i,j=1,...,m.
    \end{split}
\end{equation}
The Clifford algebra $\mathbb{R}_m$ is the vector space of dimension $2^m$ generated by 
\begin{equation*}
    \begin{split}
&\{e_0;e_1,...,e_m;e_1\cdot e_2,\dots;e_1\cdot e_2\cdot e_3,\dots;\dots;e_1\cdot...\cdot e_m\}\\
&=\{e_A=e_{\{i_1,...,i_k\}}: A=\{i_1,...,i_k\}\in\mathcal{P}(\{1,...,m\}), 1\leq i_1<\dots<i_k\leq m\},
    \end{split}
\end{equation*}
i.e. by all the possible ordered products of $e_0,\dots,e_m$ and it is
endowed with the product \eqref{eq product clifford}, extended by associativity and bilinearity to the all algebra. For $m>2$, $\mathbb{R}_m$ is an associative, non commutative algebra, which is not an integral domain. For $m=1$, it reduces to the complex numbers, while $\mathbb{R}_2\cong\mathbb{H}$.

Any Clifford number $x\in\mathbb{R}_m$ can be uniquely written as 
    $x=\sum_{A\in\mathcal{P}(\{1,...,m\})}x_Ae_A$,
where $x_A\in\mathbb{R}$ and if $A=\{i_1,...,i_k\}$, with $1\leq i_1<\dots<i_k$, $e_A\coloneqq e_{i_1}\cdot\dots\cdot e_{i_k}$. We write $e_0=e_\emptyset=1$, the unity of the algebra.
We can decompose the Clifford algebra $\mathbb{R}_m$ as
$\mathbb{R}_m=\bigoplus_{k=1}^{m}\mathbb{R}_m^k$,
where $\mathbb{R}_m^k=\{x=[x]_k=x_Ae_A\in\mathbb{R}_m:|A|=k\}$. According to this decomposition, any Clifford number can be respresented as $x=[x]_0+[x]_1+\dots+[x]_m$.
Elements contained in $\mathbb{R}_m^0=\spn(e_0)\eqqcolon\mathbb{R}$ or $\mathbb{R}_m^1=\spn(e_1,\dots,e_m)$ will be called real numbers and vectors, respectively. Elements of the form $x=x_0+\sum_{|A|=1}x_Ae_A=x_0+\sum_{j=1}^mx_je_j$ belong to $\mathbb{R}_m^0\oplus\mathbb{R}_m^1=\spn(e_0,e_1,\dots,e_m)$ and are called paravectors. The paravector subspace is isomorphic to $\mathbb{R}^{m+1}$ by the isomorphism 
    $\mathbb{R}^{m+1}\ni(x_0,x_1,\dots,x_m)\mapsto x_0+\sum_{j=1}^mx_je_j\in\mathbb{R}_m^0\oplus\mathbb{R}_m^1$.
For this reason, we will simply denote the subspace of paravectors with $\mathbb{R}^{m+1}$.

We can define a conjugations on Clifford numbers: if $x=[x]_0+[x]_1+\dots+[x]_m$, then 
\begin{equation}
\label{eq clifford conjugation}
    \overline{x}=[x]_0-[x]_1-[x]_2+[x]_3+[x]_4-\dots=\sum_{j=1}^m(-1)^{\frac{j(j+1)}{2}}[x]_j.
\end{equation}
This makes $\mathbb{R}_m$ a real associative $\ast$-algebra.
We can select the set of imaginary units in $\mathbb{R}^{m+1}$ as
\begin{equation*}
    \mathbb{S}\coloneqq\{x\in \mathbb{R}^{m+1} : x_0=0, x_1^2+\dots+x_m^2=1\}=\{x\in \mathbb{R}^{m+1} : \overline{x}=-x,\,x^2=-1\}.
\end{equation*}
For any $J\in\mathbb{S}$, the subspace $\mathbb{C}_J\coloneqq\spn(1,J)$ is a $^\ast$-algebra isomorphic to $\mathbb{C}$, via the $^\ast$-algebra isomorphism 
\begin{equation}
\label{eq isomorphism phi J}
    \phi_J:\mathbb{C}\to\mathbb{C}_J,\qquad \phi_J(a+ib)\coloneqq a+Jb.
\end{equation}
In particular, for any $x\in \mathbb{R}^{m+1}\setminus\mathbb{R}$, there exist unique $\alpha,\beta\in\mathbb{R}$, with $\beta>0$ and $J\in\mathbb{S}$ such that $x=\alpha+J\beta$. We will use this representation several times throughout the paper.

We now give the definition of monogenic function.
\begin{defn}
    Let $\Omega\subset\mathbb{R}^{m+1}$ be an open set and let 
    \begin{equation}
    \label{eq defn dirac operator}
\partial\coloneqq\frac{1}{2}\left(\frac{\partial}{\partial x_0}-\sum_{i=1}^ke_i\frac{\partial}{\partial x_i}\right),\qquad\overline{\partial}\coloneqq\frac{1}{2}\left(\frac{\partial}{\partial x_0}+\sum_{i=1}^ke_i\frac{\partial}{\partial x_i}\right).
    \end{equation}
    A differentiable function $f:\Omega\to\mathbb{R}_m$ is called monogenic if $\overline{\partial}f=0$. We will denote by $\mathcal{M}(\Omega)$ the set of monogenic functions with domain $\Omega$. 
The importance of these operators 
is evident as they
factorize the Laplacian, indeed
\begin{equation}
\label{equazione fattorizzazione laplaciano}
   4\partial\overline{\partial}=4\overline{\partial}\partial=\Delta_{m+1}.
\end{equation}
\end{defn}


\subsection{One Clifford variable slice functions}
\label{One variable theory}
We present the theory of slice regular functions of one variable as in \cite{SRFonAA}, adjusting it to Clifford algebras.

Let $\{1,e_1\}$ denote a basis of $\mathbb{R}^2$. Consider the algebra $\mathbb{R}_m\otimes\mathbb{R}^2=\{a+e_1b:a,b\in \mathbb{R}_m\}$, where $1$ is the unity of $\mathbb{R}_m\otimes\mathbb{R}^2$ and $e_1^2=-1$, thus the product of any elements of $\mathbb{R}_m\otimes\mathbb{R}^2$ is defined by bilinearity as
\begin{equation*}
    (a+e_1b)(\alpha+e_1\beta)=a\alpha-b\beta+e_1(a\beta+b\alpha),
\end{equation*}
where $a\alpha$ is the product of $\mathbb{R}_m$, whenever $a,\alpha\in \mathbb{R}_m$. Equip $\mathbb{R}_m\otimes\mathbb{R}^2$ with the conjugation
    $\overline{a+e_1b}=a-e_1b$.
This makes $(\mathbb{R}_m\otimes\mathbb{R}^2,\overline{\phantom{x}})$ a $\ast$-algebra, too.
\begin{defn}
    A set $D\subset\mathbb{C}$ is called symmetric if it is invariant with respect to conjugation, which means that
$z\in D\iff\overline{z}\in D$.
  Assume $D$ open, as well. A function $F:D\to \mathbb{R}_m\otimes\mathbb{R}^2$ is called stem function if it is complex instrinsic, i.e. it satisfies
\begin{equation}
\label{eq defn stem functions}
    F(\overline{z})=\overline{F(z)},\qquad\forall z\in D.
\end{equation}
If $F=F_0+e_1F_1$, with $F_0,F_1:D\to \mathbb{R}_m$, $F$ is a stem function if, and only if,
\begin{equation*}
    F_0(\overline{z})=F_0(z),\quad F_1(\overline{z})=-F_1(z),\qquad \forall z\in D.
\end{equation*}
The set of stem functions over $D$ is denoted by $Stem(D)$.
\end{defn}

Left multiplication by $e_1$ defines a complex structure on $\mathbb{R}_m\otimes\mathbb{R}^2$. Given a stem function $F\in\mathcal{C}^1(D)$, consider the following Wirtinger operators
\begin{equation*}
    \frac{\partial F}{\partial z}=\frac{1}{2}\left(\frac{\partial F}{\partial\alpha}-e_1\frac{\partial F}{\partial\beta}\right),\qquad\frac{\partial F}{\partial\overline{z}}=\frac{1}{2}\left(\frac{\partial F}{\partial\alpha}+e_1\frac{\partial F}{\partial\beta}\right).
\end{equation*}
    A stem function is said to be holomorphic if $F\in\ker(\partial/\partial\overline{z})$. This is equivalent to require its components $F_0,F_1$ satisfy the following Cauchy-Riemann equations:
\begin{equation*}
    \frac{\partial F_0}{\partial\alpha}=\frac{\partial F_1}{\partial\beta},\qquad \frac{\partial F_0}{\partial\beta}=-\frac{\partial F_1}{\partial\alpha}.
\end{equation*}

\begin{defn}
    Given a symmetric set $D\subset \mathbb{C}$, we define its circularization $\Omega_D$ in $\mathbb{R}^{m+1}$ as
\begin{equation*}
    \Omega_D\coloneqq\bigcup_{J\in\mathbb{S}}\phi_J(D)=\{\alpha+J\beta:\alpha+i\beta\in D, J\in\mathbb{S}\}\subset \mathbb{R}^{m+1}.
\end{equation*}
A set $\Omega$ is called circular, or axially symmetric, if $\Omega=\Omega_D$ for some symmetric set $D\subset \mathbb{C}$. An axially symmetric set $\Omega_D$ is called slice domain if $D\cap\mathbb{R}\neq\emptyset$ and product domain if $D\cap\mathbb{R}=\emptyset$.
\end{defn}
Every stem function $F:D\to \mathbb{R}_m\otimes\mathbb{R}^2$ induces uniquely a function $f:\Omega_D\to \mathbb{R}_m$ as follows: 
\begin{defn}
Let $F=F_0+e_1F_1:D\to \mathbb{R}_m\otimes\mathbb{R}^2$ be a stem function. We define $f:\Omega_D\to\mathbb{R}_m$ for any $x=\alpha+J\beta=\phi_J(z)\in\Omega_D$, as
\begin{equation}
\label{eq definizione slice functions}
    f(x)=F_0(z)+JF_1(z).
\end{equation}
We will say that $f$ is induced by $F$ ($f=\mathcal{I}(F)$) and such induced functions are called slice functions. 
We denote by $\mathcal{S}(\Omega_D)$ the set of slice functions over $\Omega_D$ and by
    $\mathcal{I}:Stem(D)\to\mathcal{S}(\Omega_D)$
the map sending a stem function to its induced slice function.
\end{defn}

Alternatively, one can also define slice functions through commutative diagrams. For any $J\in\mathbb{S}$, extend $\phi_J:\mathbb{R}_m\otimes\mathbb{R}^2\to\mathbb{R}_m$, $\phi_J(a+e_1 b)=a+Jb$. 
Given $F\in Stem(D)$, its induced slice function $f=\mathcal{I}(F)$ is defined as the unique slice function that makes the following diagram commutative for any $J\in\mathbb{S}$:
\begin{center}
    \begin{tikzcd}
	{D} && {\mathbb{R}_m\otimes\mathbb{R}^2} \\
	& \circlearrowleft \\
	{\Omega_D} && {A.}
	\arrow["F", from=1-1, to=1-3]
	\arrow["{\phi_J}"', from=1-1, to=3-1]
	\arrow["f"', from=3-1, to=3-3]
	\arrow["{\phi_J}", from=1-3, to=3-3]
\end{tikzcd}
\end{center}

\begin{defn}
   Let $f=\mathcal{I}(F)\in\mathcal{S}(\Omega_D)$ be a slice function. If $F$ is holomorphic, we say that $f$ is slice regular and we denote with $\mathcal{S}\mathcal{R}(\Omega_D)$ the set of slice regular functions over $\Omega_D$.
\end{defn}


Since $\partial F/\partial z$ and $\partial F/\partial\overline{z}$ are stem functions, we
 can also define slice derivatives of a slice function $f=\mathcal{I}(F)\in\mathcal{S}(\Omega_D)\cap\mathcal{C}^1(\Omega_D)$ as
\begin{equation*}
    \frac{\partial f}{\partial x}=\mathcal{I}\left(\frac{\partial F}{\partial z}\right),\qquad \frac{\partial f}{\partial x^c}=\mathcal{I}\left(\frac{\partial F}{\partial \overline{z}}\right).
\end{equation*}
In particular, a slice function $f$ is slice regular if and only if $\partial f/\partial x^c=0$.


Every slice function is uniquely determined by its value on two distinct half planes $\mathbb{C}_J^+$ and $\mathbb{C}_K^+$, if $J-K$ is invertible. 
\begin{prop}[\cite{SRFonAA}, Proposition 6]
\label{Prop representation formulas}
    Let $f\in\mathcal{S}(\Omega_D)$, define its restriction on the complex half plane $f^+_J\coloneqq f|_{\mathbb{C}^+_J\cap\Omega_D}$, then, for any $x=\alpha+I\beta\in\Omega_D$ we have
\begin{equation*}
    f(\alpha+I\beta)=(I-K)(J-K)^{-1}f^+_J(\alpha+J\beta)-(I-J)(J-K)^{-1}f^+_K(\alpha+K\beta).
\end{equation*}
\end{prop}

\begin{imposs}
The theory of slice functions over a real alternative $^\ast$-algebras $A$ is defined on open symmetric sets of the quadratic cone $Q_A$ (\cite{SRFonAA}). In the specific context of Clifford algebras, it is common to consider the domain of slice functions restricted to the paravector subspace $\mathbb{R}^{m+1}$, as in the theory of monogenic functions. Fortunately, the paravector subspace is always contained in the quadratic cone $Q_{\mathbb{R}_m}$ and thanks to Proposition \ref{Prop representation formulas}, the restriction to $\mathbb{R}^{m+1}$ uniquely determines the slice function. 
\end{imposs}

The previous formula is known as representation formula. For $I=J=-K$, it reduces to
\begin{equation*}
    f(x)=\frac{1}{2}\left(f(x)+f(\overline{x})\right)+\frac{1}{2}\left(f(x)-f(\overline{x})\right),
\end{equation*}
with $x=\alpha+J\beta$ and $\overline{x}=\alpha-J\beta$. If $x=\phi_J(z)$, it holds
    $\frac{1}{2}\left(f(x)+f(\overline{x})\right)=F_0(z) $
and if $\operatorname{Im}(x)\neq0$,
    $[2\operatorname{Im}(x)]^{-1}\left(f(x)-f(\overline{x})\right)=F_1(z)/\beta$,
where $F_0(z)$ and $F_1(z)/\beta$ are $ \mathbb{R}_m$-valued stem functions. This leads to the following
\begin{defn}
Given $f=\mathcal{I}(F)\in\mathcal{S}(\Omega_D)$, with $F=F_0+e_1F_1$, the spherical value and the spherical derivative of $f$ are defined respectively as
\begin{equation*}
    \begin{split}
f^\circ_s(x)&\coloneqq\mathcal{I}(F_0)(x)=\frac{1}{2}\left(f(x)+f(\overline{x})\right),\qquad\forall x\in\Omega_D\\
f'_s(x)&\coloneqq\mathcal{I}(F_1/\operatorname{Im}(z))(x)=[2\operatorname{Im}(x)]^{-1}\left(f(x)-f(\overline{x})\right),\qquad\forall x\in\Omega_D\setminus\mathbb{R}.
    \end{split}
\end{equation*}
Moreover, they decompose $f$ through
\begin{equation*}
    f(x)=f^\circ(x)+\operatorname{Im}(x)f'_s(x).
\end{equation*}
\end{defn}
Since $F_0(z)$ and $F_1(z)/\operatorname{Im}(z)$ are $ \mathbb{R}_m$-valued,
$f^\circ_s(x)$ and $f'_s(x)$ depends only on $\operatorname{Re}(x)$ and $|\operatorname{Im}(x)|$. This means that $f^\circ_s$ and $f'_s$ are constant on every sphere $\mathbb{S}_{\alpha,\beta}=\{\alpha+I\beta: I\in\mathbb{S}\}$. This subclass of functions deserves a name.

\begin{defn}
    Let $f=\mathcal{I}(F)$ be a slice function. Suppose that $F=F_0$, i.e. $F$ is an $ \mathbb{R}_m$-valued stem function. Then we call $f$ a circular slice function. Namely, circular slice functions are slice functions which are constant over "spheres" $\mathbb{S}_{\alpha,\beta}=\{\alpha+I\beta: I\in\mathbb{S}\}$.
\end{defn}

\begin{prop}
    Let $f$ be a circular slice function. Then $f^\circ_s=f$ and $f'_s=0$. In particular, for any slice function $f$ it holds $(f'_s)'_s=(f^\circ_s)'_s=0$, $(f^\circ_s)^\circ_s=f^\circ_s$ and $(f'_s)^\circ_s=f'_s$.
\end{prop}

We can give stem functions and slice functions the structure of algebras, by defining a product of stem functions, that  will induce one on slice functions. 
\begin{defn}
    Let $F,G\in Stem(D)$, with $F=F_0+e_1 F_1$ and $G=G_\emptyset+e_1G_1$. Define $$F\otimes G:=F_0 G_\emptyset-F_1G_1+e_1(F_0 G_1+F_1G_0).$$
It is easy to prove that $F\otimes G$ is a stem function. Now, if $f=\mathcal{I}(F)$ and $g=\mathcal{I}(G)$, define $f\odot g\coloneqq\mathcal{I}(F\otimes G).$
\end{defn}
As a consequence, $(Stem(D),\otimes)$ and $(\mathcal{S}(\Omega_D),\odot)$ form an algebra and $\mathcal{I}:(Stem(D),\otimes)\to (\mathcal{S}(\Omega_D),\odot)$ is an algebra isomorphism. Observe tha, in general, $(f\odot g)(x)\neq f(x)g(x)$. With respect to this product, the spherical derivative satisfies a Lebniz rule, in which evaluation is replaced by spherical value:
    $(f\odot g)'_s=f'_s\odot g^\circ_s+f^\circ_s\odot g'_s$.

\subsection{Several Clifford variables slice functions}
We follow \cite{Several} for the several variables version of the theory of slice regular functions. We restrict our attention to Clifford algebras, which are associative.  

Let $n\in\mathbb{N}^*$ be a positive integer and let $\mathcal{P}(n)\coloneqq\mathcal{P}(\{1,...,n\})$ denote all possible subsets of $\{1,...,n\}$. Given a sequence $x=(x_1,\dots x_n)\in (\mathbb{R}_m)^n$, define its ordered product as $[x]=x_1$ if $n=1$ and for $n\geq 2$,
\begin{equation*}
    [x]=[x_1,\dots x_n]=x_1\cdot x_2\cdots x_{n-1}\cdot x_n.
\end{equation*}
Moreover, given $y\in \mathbb{R}_m$, we denote
\begin{equation*}
    [x,y]=[x_1,\dots,x_n,y]=x_1\cdot x_2\cdots x_{n-1}\cdot x_{n}\cdot y.
\end{equation*}
Let $K=\{k_1,...,k_p\}\in\mathcal{P}(n)$ be an ordered set of indexes, with $k_1<\dots<k_p$. If $K=\emptyset$, then set $x_K=\emptyset$ and $[x_K]=1$; if $K\neq 0$, define $x_K=(x_{k_1},\dots,x_{k_p})\in (\mathbb{R}_m)^p$, so by the definition above
\begin{equation*}
    [x_K]=[x_{k_1},\dots x_{k_p}]=x_{k_1}\cdot x_{k_2}\cdots x_{k_{p-1}}\cdot x_{k_p}
\end{equation*}
and
\begin{equation*}
    [x_K,y]=[x_{k_1},\dots x_{k_p},y]=x_{k_1}\cdot x_{k_2}\cdots x_{k_{p-1}}\cdot x_{k_p}\cdot y.
\end{equation*}
\begin{defn}
    Given $z=(z_1,\dots,z_n)\in\mathbb{C}^n$, define $\overline{z}^h\coloneqq(z_1,\dots,z_{h-1},\overline{z}_h,z_{h+1},\dots,z_n)$, for any $h\in\{1,...,n\}$.
A set $D\subset\mathbb{C}^n$ is called symmetric if it is invariant with respect to complex conjugation in any variable, i.e. if $z\in D\iff\overline{z}^h\in D$, for every $ h=1,...,n$. 
\end{defn}
Let $\{e_1,\dots,e_n\}$ be an orthonormal basis of $\mathbb{R}^n$ and denote with $\{e_K\}_{K\in\mathcal{P}(n)}$ a basis of $\mathbb{R}^{2^n}$. 
\begin{defn}
    Let $D\subset\mathbb{C}^n$ be an open symmetric set and consider a function $F:D\subset\mathbb{C}^n\to \mathbb{R}_m\otimes\mathbb{R}^{2^n}$, $F(z)=\sum_{K\in\mathcal{P}(n)}e_KF_K(z)$ with $F_K:D\to \mathbb{R}_m$. We call $F$ a stem function if $F_K(\overline{z}^h)=(-1)^{|K\cap\{h\}|}F_K(z)$ or equivalently
\begin{equation}
\label{eq properies stem functions sev var}
    F_K(\overline{z}^h)=\left\{
\begin{array}{ll}
     F_K(z)&\text{ if }h\notin K  \\
    -F_K(z)&\text{ if }h\in K,
\end{array}
    \right.
\end{equation}
for every $z\in D$, every $K\in\mathcal{P}(n)$ and any $h\in\{1,\dots,n\}$. Again, we use the symbol $Stem(D)$ to denote the set of stem functions $F:D\to \mathbb{R}_m\otimes\mathbb{R}^{2^n}$.
\end{defn}
Equip $\mathbb{R}^{2^n}$ with the family of commutative complex structures $\mathcal{J}=\left\{\mathcal{J}_h:\mathbb{R}^{2^n}\rightarrow\mathbb{R}^{2^n}\right\}_{h=1}^n,$
where each $\mathcal{J}_h$ is defined over any basis element $e_K$ of $\mathbb{R}^{2^n}$ as
\begin{equation*}
    \mathcal{J}_h(e_K):=(-1)^{|K\cap\{h\}|}e_{K\Delta\{h\}}=\left\{\begin{array}{ll}
e_{K\cup\{h\}} &\text{ if }h\notin K  \\
-e_{K\setminus\{h\}} &\text{ if }h\in K,
    \end{array}\right.
\end{equation*}
where $K\Delta H=(K\cup H)\setminus(K\cap H)$
and extend it by linearity to all $\mathbb{R}^{2^n}$. $\mathcal{J}$ induces a family of commutative complex structure on $\mathbb{R}_m\otimes\mathbb{R}^{2^n}$ (by abuse of notation, we use the same symbol) 
$\mathcal{J}=\left\{\mathcal{J}_h: \mathbb{R}_m\otimes\mathbb{R}^{2^n}\rightarrow \mathbb{R}_m\otimes\mathbb{R}^{2^n}\right\}_{h=1}^n$ according to the formula $$\mathcal{J}_h(x\otimes a):=x\otimes\mathcal{J}_h(a)\qquad\forall x\in \mathbb{R}_m,\quad\forall a\in\mathbb{R}^{2^n}.$$
We can associate two Cauchy-Riemann operators to each complex structure $\mathcal{J}_h$. 
\begin{defn}
    Given a stem function $F\in Stem(D)\cap\mathcal{C}^1(D)$, we define
\begin{equation*}
    \partial_hF:=\dfrac{1}{2}\left(\dfrac{\partial F}{\partial \alpha_h}-\mathcal{J}_h\left(\dfrac{\partial F}{\partial\beta_h}\right)\right),\qquad \overline{\partial}_hF:=\dfrac{1}{2}\left(\dfrac{\partial F}{\partial \alpha_h}+\mathcal{J}_h\left(\dfrac{\partial F}{\partial\beta_h}\right)\right).
\end{equation*}
We call $F=\sum_{K\in\mathcal{P}(n)}e_KF_K$ $h$-holomorphic (with respect to $\mathcal{J}$) if $F\in\ker\overline{\partial}_h$ and it is called holomorphic if it is $h$-holomorphic for every $h=1,...,n$.
\end{defn}
We can give the definition of holomorphic stem function through a system of Cauchy-Riemann equations.
\begin{prop}[\cite{Several},Lemma 3.12]
    Let $F$ be a stem function. Then $F$ is $h$-holomorphic if and only if
    \begin{equation}
    \label{Cauchy-Riemann equations}
    \dfrac{\partial F_K}{\partial\alpha_h}=    \dfrac{\partial F_{K\cup\{h\}}}{\partial\beta_h},\qquad
    \dfrac{\partial F_K}{\partial\beta_h}=-    \dfrac{\partial F_{K\cup\{h\}}}{\partial\alpha_h},\qquad\forall K\in\mathcal{P}(n), h\notin K.
\end{equation}
\end{prop}

For any $J_1,\dots J_n\in\mathbb{S}$, define 
\begin{equation*}
\phi_{J_1}\times...\times\phi_{J_n}:\mathbb{C}^n\ni(z_1,...,z_n)\mapsto\left(\phi_{J_1}(z_1),...,\phi_{J_n}(z_n)\right)\in(\mathbb{R}^{m+1})^n,
\end{equation*}
where $\phi_J$ is defined in \eqref{eq isomorphism phi J}.

\begin{defn}
    Given a symmetric set $D\subset\mathbb{C}^n$, we define its circularization $\Omega_D\subset(\mathbb{R}^{m+1})^n$ as $\Omega_D\coloneqq\bigcup_{(J_1,\dots,J_n)\in\mathbb{S}^n}(\phi_{J_1}\times\dots\times\phi_{J_n})(D)$, namely
\begin{equation*}
    \Omega_D=\{(\alpha_1+J_1\beta_1,\dots,\alpha_n+J_n\beta_n):(\alpha_1+i\beta_1,\dots,\alpha_n+i\beta_n)\in D, J_1,\dots,J_n\in\mathbb{S}\}.
\end{equation*}
A subset $\Omega\subset (\mathbb{R}^{m+1})^n$ is called circular if it is the circularization of a symmetric set $D$, i.e. $\Omega=\Omega_D$ for some $D\subset\mathbb{C}^n$. 
\end{defn}

\begin{defn}
\label{defn slice functions}
    A map $f:\Omega_D\subset(\mathbb{R}^{m+1})^n\to \mathbb{R}_m$ is called a slice function if there exist a stem function $F=\sum_{K\in\mathcal{P}(n)}e_KF_K:D\to \mathbb{R}_m\otimes \mathbb{R}^{2^n}$, such that, for every $x\in\Omega_D$
\begin{equation*}
    f(x)=\sum_{K\in\mathcal{P}(n)}[J_K,F_K(z)],
\end{equation*}
where $x=(\phi_{J_1}\times\dots\times\phi_{J_n})(z)$ and $J=(J_1,\dots,J_n)$. A slice regular function is a slice function induced by a holomorphic stem function. A slice function is called slice preserving whenever the components of its inducing stem function are real valued.
  We denote with $\mathcal{S}(\Omega_D)$, $\mathcal{S}\mathcal{R}(\Omega_D)$ and $\mathcal{S}_\mathbb{R}(\Omega_D)$ respectively the set of slice, slice regular and slice preserving functions on $\Omega_D\subset(\mathbb{R}^{m+1})^n$. Again, $\mathcal{I}:Stem(D)\rightarrow \mathcal{S}(\Omega_D)$ will be the map sending a stem function to its induced slice function. 
\end{defn}

Equivalently, we can define $f$ as the unique slice function that makes the following diagram commutative for any $J_1,...,J_n\in\mathbb{S}$:
\begin{center}
    \begin{tikzcd}
	{D} && {\mathbb{R}_m\otimes\mathbb{R}^{2^n}} \\
	& \circlearrowleft \\
	{\Omega_D} && {\mathbb{R}_m,}
	\arrow["F", from=1-1, to=1-3]
	\arrow["{\phi_{J_1}\times...\times\phi_{J_n}}"', from=1-1, to=3-1]
	\arrow["f"', from=3-1, to=3-3]
	\arrow["{\Phi_{J_1,...,J_n}}", from=1-3, to=3-3]
\end{tikzcd}
\end{center}
where $$\Phi_{J_1,...,J_n}:\mathbb{R}_m\otimes\mathbb{R}^{2^n}\ni \sum_{K\in\mathcal{P}(n)}e_Ka_K\mapsto \sum_{K\in\mathcal{P}(n)}\left[J_K,a_K\right]\in \mathbb{R}_m.$$

 
 If $F$ is a stem function, so are $\partial_hF$ and $\overline{\partial}_hF$ \cite[Lemma 3.9]{Several}, thus, if $f=\mathcal{I}(F)\in\mathcal{S}^1(\Omega_D):=\mathcal{I}(Stem(D)\cap\mathcal{C}^1(D))$, we can define the partial derivatives for every $h=1,...,n$
\begin{equation*}
    \dfrac{\partial f}{\partial x_h}:=\mathcal{I}\left(\partial_hF\right),\qquad \dfrac{\partial f}{\partial  x^c_h}:=\mathcal{I}\left(\overline{\partial}_hF\right).
\end{equation*}
In particular, $f\in\mathcal{S}\mathcal{R}(\Omega_D)$ if and only if $\frac{\partial f}{\partial  x^c_h}=0$ for every $h=1,...,n$.

Equip $\mathbb{R}^{2^n}$ with the product  $\otimes:\mathbb{R}^{2^n}\times\mathbb{R}^{2^n}\rightarrow\mathbb{R}^{2^n}$, defined on each basis element as
\begin{equation*}
    e_H\otimes e_K:=(-1)^{|H\cap K|}e_{H\Delta K},
\end{equation*}
and extended by linearity to all $\mathbb{R}^{2^n}$. This product induces a product on $\mathbb{R}_m\otimes\mathbb{R}^{2^n}$: given $a,b\in \mathbb{R}_m\otimes\mathbb{R}^{2^n}$, $a=\sum_{H\in\mathcal{P}(n)}e_Ha_H$ and $b=\sum_{K\in\mathcal{P}(n)}e_Kb_K$, with $a_H,b_K\in \mathbb{R}_m$, define
\begin{equation*}
    a\otimes b:= \sum_{H,K\in\mathcal{P}(n)}(e_H\otimes e_K)(a_Hb_K)= \sum_{H,K\in\mathcal{P}(n)}(-1)^{|H\cap K|}e_{H\Delta K}a_Hb_K,
\end{equation*}
where $a_Hb_K$ is the usual product of $\mathbb{R}_m$.
Furthermore, we can define a product between stem functions as the pointwise product induced by $\otimes$.

\begin{defn}
    Let $F,G\in Stem(D)$, define $(F\otimes G)(z):=F(z)\otimes G(z)$. More precisely, if $F=\sum_{H\in\mathcal{P}(n)}e_HF_H$ and $G=\sum_{K\in\mathcal{P}(n)}e_KG_K$,
\begin{equation*}
    (F\otimes G)(z):= \sum_{H,K\in\mathcal{P}(n)}(-1)^{|H\cap K|}e_{H\Delta K}F_H(z)G_K(z).
\end{equation*}
The advantage of this definition is that the product of two stem functions is again a stem function \cite[Lemma 2.34]{Several} and this allows to define a product on slice functions, too. Let $f,g\in\mathcal{S}(\Omega_D)$, with $f=\mathcal{I}(F)$ and $g=\mathcal{I}(G)$, then define the slice tensor product between $f$ and $g$ as 
    $f\odot g:=\mathcal{I}(F\otimes G)$.
In particular, $\mathcal{I}:(Stem(D),\otimes)\to(\mathcal{S}(\Omega_D),\odot)$ is an algebra isomorphism.
\end{defn}

\subsection{Partial slice regularity}
\label{Sezione proprietà sliceness rispetto a variabile}
The notion of partial sliceness was already given in \cite{Several}. The results of this section are taken from \cite{Parteteorica}, where they were proven for quaternionic valued slice functions. The very same proofs apply to the Clifford algebra setting.

Let $f:\Omega_D\subset (\mathbb{R}^{m+1})^n\rightarrow \mathbb{R}^{m+1}$ and $h=1,...,n$. For any $y=(y_1,...,y_n)\in\Omega_D$, let 
$$\Omega_{D,h}(y):=\{x\in \mathbb{R}_m\mid (y_1,...,y_{h-1},x,y_{h+1},...,y_n)\in\Omega_D\}\subset \mathbb{R}^{m+1}.$$
It is easy to see (\cite[\S 2]{Several}) that $\Omega_{D,h}(y)$ is a circular set of $ \mathbb{R}^{m+1}$, more precisely $\Omega_{D,h}(y)=\Omega_{D_h(z)}$, where
    $D_h(z):=\{w\in\mathbb{C}\mid (z_1,...,z_{h-1},w,z_{h+1},...,z_n)\in D\}$,
for $z=(z_1,...,z_n)$, such that $y\in\Omega_{\{z\}}$. 

\begin{defn}
We say that a slice function $f\in\mathcal{S}(\Omega_D)$ is \emph{slice} (resp. \emph{slice regular} or \emph{circular}) \emph{with respect to} $x_h$ if, $\forall y\in\Omega_D$, its restriction
\begin{equation*}
    f^y_h:\Omega_{D,h}(y)\rightarrow \mathbb{R}_m, \ f^y_h(x):=f(y_1,...,y_{h-1},x,y_{h+1},...,y_n)
\end{equation*}
is a one variable slice (resp. slice regular or circular) function, as defined in \S\ref{One variable theory}. 
We denote by $\mathcal{S}_h(\Omega_D)$ (resp. $\mathcal{S}\mathcal{R}_h(\Omega_D)$ or $\mathcal{S}_{c,h}(\Omega_D)$) the set of slice functions from $\Omega_D$ to $ \mathbb{R}_m$ that are slice (resp. slice regular or circular) with respect to $x_h$. For $H\in\mathcal{P}(n)$, define also
\begin{equation*}
    \mathcal{S}_H(\Omega_D):=\bigcap_{h\in H}\mathcal{S}_h(\Omega_D),\quad \mathcal{S}\mathcal{R}_H(\Omega_D):=\bigcap_{h\in H}\mathcal{S}\mathcal{R}_h(\Omega_D),\quad \mathcal{S}_{c,H}(\Omega_D):=\bigcap_{h\in H}\mathcal{S}_{c,h}(\Omega_D).
\end{equation*}
\end{defn}

Every slice function is, in particular, slice with respect to the first variable \cite[Proposition 2.23]{Several}, i.e. $\mathcal{S}_1(\Omega_D)=\mathcal{S}(\Omega_D)$, but in general $\mathcal{S}_h(\Omega_D)\subsetneq \mathcal{S}(\Omega_D)$. The same happens for slice regularity. The next proposition characterizes the sets $\mathcal{S}_H(\Omega_D),\mathcal{SR}_H(\Omega_D)$ and $\mathcal{S}_{c,H}(\Omega_D)$ for any $H\in\mathcal{P}(n)$ in terms of stem functions.

\begin{prop}
For any $H\in\mathcal{P}(n)$, it holds
\begin{equation*}
    \mathcal{S}_H(\Omega_D)=\left\{\mathcal{I}(F): F\in Stem(D), F=\sum_{K\subset H^c}e_KF_K+\sum_{h\in H}e_{\{h\}}\sum_{Q\subset\{ h+1,...,n\}\setminus H}e_QF_{\{h\}\cup Q}\right\},
\end{equation*}
\begin{equation}
    \label{Equazione caratterizzazione slice regular H}
\mathcal{S}\mathcal{R}_H(\Omega_D)=\mathcal{S}_H(\Omega_D)\cap\bigcap_{h\in H}\ker(\partial/\partial x_h^c)\subset\mathcal{S}_H(\Omega_D),
\end{equation}
\begin{equation}\label{equazione circolarita}
    \mathcal{S}_{c,H}(\Omega_D)=\left\{\mathcal{I}(F): F\in Stem(D), F=\sum_{K\subset H^c}e_KF_K\right\}\subset\mathcal{S}_H(\Omega_D).
\end{equation}
Moreover, for every $H\in\mathcal{P}(n)$, the set $\mathcal{S}_{c,H}(\Omega_D)$ is a real subalgebra of $(\mathcal{S}(\Omega_D),\odot)$. Any $f\in\mathcal{S}_{c,h}(\Omega_D)\cap\mathcal{S}\mathcal{R}_h(\Omega_D)$ is locally constant with respect to $x_h$. Finally, if $f\in\mathcal{S}\mathcal{R}(\Omega_D)$, $f\in\mathcal{S}_H(\Omega_D)$ if and only if $ f\in\mathcal{S}\mathcal{R}_H(\Omega_D).$
\end{prop}

\begin{oss}
\label{remark structure slice function wrt xh}
Let $f=\mathcal{I}(F)\in\mathcal{S}_H(\Omega_D)$, then for any $x\in\Omega_D$ with $x=(\phi_{J_1}\times...\times \phi_{J_n})(z)$
\begin{equation*}
    f(x)= \sum_{K\in H^c}\left[J_K,F_K(z)\right]+ \sum_{h\in H}J_h \sum_{Q\subset\{h+1,...,n\}\setminus H}\left[J_Q,F_{\{h\}\cup Q}(z)\right].
\end{equation*}
Moreover, for any $h\in H$ and any $y=(y_1,...,y_n)$, $f^y_h$ is a one-variable slice function, induced by the stem function $G^y_h$, with components
\begin{equation}
\label{Equazione componenti stem parziale}
G^y_{1,h}(w):= \sum_{K\in\mathcal{P}(n),h\notin K}[J_K,F_K(z',w,z'')]
    ,\qquad
G^y_{2,h}(w):= \sum_{Q\subset\{h+1,...,n\}\setminus H}[J_Q,F_{\{h\}\cup Q}(z',w,z'')],
    \end{equation}
    where $z=(z',z_h,z")$ and $y=(\phi_{J_1}\times...\times\phi_{J_n})(z)$.
\end{oss}

Functions of the form \eqref{equazione circolarita} were introduced in \cite{Several} as $H^c$-reduced slice functions, hence $f\in\mathcal{S}_{c,H}(\Omega_D)$ if and only if it is $H^c$-reduced.

For $h\in\{1,...,n\}$, define $\mathbb{R}_h:=\{(x_1,...,x_n)\mid x_h\in\mathbb{R}\}$ and for $H\in\mathcal{P}(n)$, $\mathbb{R}_H:=\bigcup_{h\in H}\mathbb{R}_h$.
\begin{defn}
    Let $F:D\subset\mathbb{C}^n\rightarrow \mathbb{R}_m\otimes\mathbb{R}^{2^n}$ be a stem function. Define for $h=1,...,n$ and for $H=\{h_1,...,h_p\}\in\mathcal{P}(n)$,
\begin{equation*}
\begin{split}
    F^\circ_h(z):= \sum_{K\in\mathcal{P}(n),h\notin K}e_KF_K(z),\qquad
F^\circ_H(z):=\sum_{K\subset H^c}e_KF_K(z)=\left(\dots(F^\circ_{h_1})^\circ_{h_2}\dots\right)^\circ_{h_p}(z)
\end{split}
    \end{equation*}
    and 
    \begin{align}
F'_h(z):&=\sum_{K\in\mathcal{P}(n),h\notin K}e_K\beta_h^{-1}F_{K\cup\{h\}}(z),&\text{if }z\in D\setminus\mathbb{R}_h\\
F'_H(z):&=\sum_{K\subset H^c}e_K\beta_H^{-1}F_{K\cup H}(z)=\left(\dots(F'_{h_1})'_{h_2}\dots\right)'_{h_p}(z), &\text{if }z\in D\setminus\mathbb{R}_H,
    \end{align}
    where $z=(z_1,...,z_n)$ with $z_j=\alpha_j+i\beta_j$
    and $\beta_H=\prod_{h\in H}\beta_h$.
    
    \end{defn}

    

Since, for every $H\in\mathcal{P}(n)$, $F^\circ_H$ and $F'_H$ are well defined stem functions we can make the following

\begin{defn}
    Let $f=\mathcal{I}(F)\in\mathcal{S}(\Omega_D)$. For $h\in\{1,...,n\}$, we define its \emph{spherical $x_h$-value and $x_h$-derivative} rispectively as
       $f^\circ_{s, h}:=\mathcal{I}(F^\circ_{h})$ and  $f'_{s, h}:=\mathcal{I}(F'_h)$.
Analogously, for $H\in\mathcal{P}(n)$, define
$f^\circ_{s,H}:=\mathcal{I}(F^\circ_H)$ and  $f'_{s,H}:=\mathcal{I}(F'_H)$.
Note that $f^\circ_{s,H}\in\mathcal{S}(\Omega_D)$, while $f'_{s,H}\in\mathcal{S}(\Omega_{D_H})$, where $\Omega_{D_H}:=\Omega_D\setminus\mathbb{R}_H$.
\end{defn}
For any $f\in\mathcal{S}(\Omega_D)$ it holds $f^\circ_{s, h}(x)=\dfrac{1}{2}\left(f(x)+f\left(\overline{x}^h\right)\right)=(f^x_h)^\circ_s(x_h)$, while if $f\in\mathcal{S}_h(\Omega_D)$, then 
\begin{equation}\label{equazione derivata sferica parziale coincide con unidimensionale}
    f'_{s, h}(x)=\left[2\operatorname{Im}(x_h)\right]^{-1}(f(x)-f(\overline{x}^h))=(f^x_h)'_{s}(x_h)
\end{equation}
for any $x\in\Omega_D\setminus\mathbb{R}$.





The next proposition presents some properties of partial spherical values and derivatives peculiar of the several variables setting.
\begin{prop}
\label{proposizione proprieta derivata sferica}
Let $f,g\in\mathcal{S}(\Omega_D)$, $h\in\{1,...,n\}$ and $H\in\mathcal{P}(n)$, with $p=\min H^c$ if $H\neq\{1,...,n\}$. Then
\begin{enumerate}
    \item $f^\circ_{s,H}\in\mathcal{S}_{c,H}(\Omega_D)\cap\mathcal{S}_p(\Omega_D)$ and $f'_{s,H}\in\mathcal{S}_{c,H}(\Omega_{D_H})\cap\mathcal{S}_p(\Omega_{D_H})$;
    \item if $f\in\mathcal{S}_h(\Omega_{D})$, $f'_{s, h}\in\mathcal{S}_{h+1}(\Omega_{D_H})\cap\mathcal{S}_{c,\{1,...,h\}}(\Omega_{D_H})$;
    \item if $f\in\mathcal{S}_{c,h}(\Omega_D)$, $f^\circ_{s,h}=f$ and $f'_{s,h}=0$;
    \item if $h\in H$, $H\cap\{1,...,h-1\}\neq\emptyset$ and $f\in\mathcal{S}_h(\Omega_D)$, then $f'_{s,H}=0$;
    \item $(f^\circ_{s,h})^\circ_{s,h}=f^\circ_{s,h}$ and $(f'_{s,h})'_{s,h}=0$;
    \item if $f\in\ker(\partial/\partial x_t^c)$ for some $t=1,...,n$, then $f^\circ_{s,h},f'_{s, h}\in\ker(\partial/\partial x_t^c)$, $\forall h\neq t$;
    \item
$f=f^\circ_{s, h}+\operatorname{Im}(x_h)\odot f'_{s, h}$;
    \item
$(f\odot g)'_{s, h}=f'_{s, h}\odot g^\circ_{s, h}+f^\circ_{s, h}\odot g'_{s, h}$.
\end{enumerate}
\end{prop}

We stress that the terms spherical value and spherical derivatives have been already used in \cite[\S 2.3]{Several} in the context of slice functions of several quaternionic variables, but they refer to different objects. With respect to our definition, spherical values and derivatives are more related to the truncated spherical derivatives.

\begin{defn}[Definition 2.24,\cite{Several}]
    Let $\Omega_D\subset(\mathbb{R}^{m+1})^n$ and let $f\in\mathcal{S}(\Omega_D)$. For any $h=1,\dots, n$ and $\epsilon:\{1,\dots,h\}\to\{0,1\}$, define the truncated spherical $\epsilon$-derivative of $f$ of order $h$, $\mathcal{D}_\epsilon^h(f):\Omega_D\setminus\mathbb{R}_{\epsilon^{-1}(1)}\to \mathbb{R}_m$ as
$\mathcal{D}_\epsilon^h(f)\coloneqq\mathcal{D}_{x_h}^{\epsilon(h)}\cdots\mathcal{D}_{x_1}^{\epsilon(1)}(f)$,
    with
$\mathcal{D}_{x_l}^{0}(f)=f^\circ_{s,l},$ and $ \mathcal{D}_{x_l}^{1}(f)=f'_{s,l}$.
    Alternatively, for given $H\in\mathcal{P}(h)$, we call the truncated spherical $H$-derivative of $f$ the truncated spherical $\chi_H$-derivative of $f$, where $\chi_H$ is the characteristic function of $H$, i.e. $\chi_H(j)=0$ if $j\notin H$ and $\chi_H(j)=1$ if $j\in H$. Namely, we set 
$\mathcal{D}_H^h(f)=\mathcal{D}_{\chi_H}^h(f)$.
\end{defn}

    For any given $h=1,\dots,n$ and any $\epsilon:\{1,\dots,h\}\to\{0,1\}$, denote with $H=\epsilon^{-1}(1)$ and $K=\epsilon^{-1}(0)=\{1,\dots,h\}\setminus H$. Then it holds
$\mathcal{D}_\epsilon^h(f)=\mathcal{D}_H^h(f)=(f'_{s,H})^\circ_{s,K}$.
    Explicitly, if $f=\mathcal{I}(F)$, with $F=\sum_{K\in\mathcal{P}(n)}e_KF_K$, for every $h=1,\dots,n$ and any $H\in\mathcal{P}(H)$, it holds $\mathcal{D}^h_H=\mathcal{I}(D^h_H)$, with
    \begin{equation*}
D^h_H=\beta_H^{-1}\sum_{K\subset\{h+1,\dots,n\}}e_KF_{K\cup H}.
    \end{equation*}

From \eqref{Equazione caratterizzazione slice regular H} and Proposition \ref{proposizione proprieta derivata sferica} (6) it is easy to see that if $f\in\mathcal{S}\mathcal{R}(\Omega_D)$, then $\mathcal{D}^h_K(f)\in\mathcal{S}\mathcal{R}_{h+1}(\Omega_D)$, for any $h=1,\dots,n$ and $H\in\mathcal{P}(h)$. Next Theorem tells that also the converse holds true.
\begin{thm}[One variable characterization of slice regularity, Theorem 3.23 \cite{Several}]
\label{thm one variable characterization}
    Let $\Omega_D\subset(\mathbb{R}^{m+1})^n$ and let $f\in\mathcal{S}(\Omega_D)$. Then $f\in\mathcal{S}\mathcal{R}(\Omega_D)$ if and only if $f\in\mathcal{S}\mathcal{R}_1(\Omega_D)$ and $\mathcal{D}_{\chi_K}^h(f)\in\mathcal{S}\mathcal{R}_{h+1}(\Omega_D)$, for any $h=1,\dots,n-1$ and any $K\in\mathcal{P}(h)$.
\end{thm}

\section{Polyharmonicity in slice analysis}
\label{Sectionharmonicity}
Most of the harmonic properties presented for one variable slice regular functions are known (see \cite{Harmonicity}). However, we provide some new explicit formulas for the iterated Laplacian applied to a slice regular function and its spherical derivatives in both one and several Clifford variables. We then prove these results in several variables.

\begin{imposs}
    For the rest of the paper we assume $m$ an odd positive integer and we set $\gamma_m\coloneqq\frac{m-1}{2}$.
\end{imposs}
\subsection{Polyharmonicity of one Clifford variable slice regular functions}
In this subsection we assume $\Omega_D\subset\mathbb{R}^{m+1}$ open circular set.
 The next formulas are a slight variation of \cite[Theorem 4.1]{Harmonicity}.

\begin{prop}
\label{Prop potenza laplaciano derivata sferica}
    Let $f\in\mathcal{SR}(\Omega_D)$. Then, for any $k=1,2,\dots$, it holds
    \begin{equation}
    \label{eq 1 potenza laplaciano sferica}
    \Delta_{m+1}^kf'_s(x)=(m-3)\cdot...\cdot(m-2k-1)\sum_{j=1}^{k}a_j^{(k)}\beta^{j-2k}\partial_\beta^jf'_s(x),
\end{equation}
or equivalently, if $f=\mathcal{I}(F_0+e_1F_1)$,
\begin{equation}
\label{eq 2 potenza laplaciano sferica}
    \Delta_{m+1}^kf'_s(x)=(m-3)\cdot...\cdot(m-2k-1)\sum_{j=1}^{k+1}a_j^{(k+1)}\beta^{j-2k-2}\partial_\beta^{j-1}F_1(\operatorname{Re}(x),|\operatorname{Im}(x)|),
\end{equation}
where $\Delta_{m+1}=\partial^2_{x_0}+\sum_{j=1}^m\partial^2_{x_j}$ is the Laplacian of $\mathbb{R}^{m+1}$, $\beta=\sqrt{x_1^2+\dots+x_m^2}$ and
\begin{equation}
\label{eq definition coefficients ajk}
a_j^{(k)}:=\frac{(2k-j-1)!}{(j-1)!\,(k-j)!\,(-2)^{k-j}}.
\end{equation}
In particular, $f'_s$ is $\gamma_m$-polyharmonic, i.e.
$\Delta^{\gamma_m}_{m+1}f'_s=0$.
\end{prop}

Before giving the proof of Proposition \ref{Prop potenza laplaciano derivata sferica}, we need some preliminary results. 

\begin{lem}\label{lem:coefficienti}
    The coefficients $a_j^{(k)}$ of the previous Proposition satisfy the following relation
   for any $j=0,\dots ,k$
\begin{equation}
\label{eq sum recursion coefficient}
    a_{j+1}^{(k+1)}=\sum_{l=j}^k(-1)^{l-j}\frac{l!}{j!}a_l^{(k)}.
\end{equation}
\end{lem}
\begin{proof}
    Note that
\begin{equation*}
    a_{j+1}^{(k+1)}=\frac{(2k-j)!}{j!(k-j)!(-2)^{k-j}}=\frac{2^{j-k}(-1)^{j-k}}{j!}\frac{(2k-j)!}{(k-j)!},
\end{equation*}
while, on the other hand
\begin{equation*}
    \sum_{l=j}^k(-1)^{l-j}\frac{l!}{j!}a_l^{(k)}=\sum_{l=j}^k(-1)^{l-j}\frac{l(2k-l-1)!}{j!(k-l)!(-2)^{k-l}}=\frac{2^{-k}(-1)^{j-k}}{j!}\sum_{l=j}^k\frac{2^{l}l(2k-l-1)!}{(k-l)!}.
\end{equation*}
Thus, \eqref{eq sum recursion coefficient} holds if and only if
\begin{equation*}
    \frac{2^j(2k-j)!}{(k-j)!}=\sum_{l=j}^k\frac{2^{l}l(2k-l-1)!}{(k-l)!}.
\end{equation*}
Consider the right hand side of the previous equation
    $\sum_{l=j}^k\frac{2^{l}l(2k-l-1)!}{(k-l)!}=\sum_{l=j}^{k-1}\frac{2^{l}l(2k-l-1)!}{(k-l)!}+2^kk!$,
note that
\begin{equation*}
    \frac{2^{l}l(2k-l-1)!}{(k-l)!}=\frac{2^{l}(2k-l)!}{(k-l)!}-\frac{2^{l+1}(2k-(l+1))!}{(k-(l+1)!},
\end{equation*}
thus the sum $\sum_{l=j}^k\frac{2^{l}l(2k-l-1)!}{(k-l)!}$ is telescopic and gives
$\sum_{l=j}^k\frac{2^{l}l(2k-l-1)!}{(k-l)!}=\frac{2^j(2k-j)!}{(k-j)!}-2^kk!$.
Finally
\begin{equation*}
    \sum_{l=j}^k\frac{2^{l}l(2k-l-1)!}{(k-l)!}=\frac{2^j(2k-j)!}{(k-j)!}-2^kk!+2^kk!=\frac{2^j(2k-j)!}{(k-j)!}.
\end{equation*}
\end{proof}

\begin{proof}[Proof of Proposition \ref{Prop potenza laplaciano derivata sferica}]
\eqref{eq 1 potenza laplaciano sferica} is proven in \cite[Proposition 4.2]{SCHCF}, moreover \eqref{eq 2 potenza laplaciano sferica} follows from Lemma \ref{lem:coefficienti}. Indeed, recall that $f'_s=\beta^{-1}F_1$ and that for any $j=1,2,\dots$ it holds
\begin{equation*}
    \partial_\beta^j(\beta^{-1}F_1)=\sum_{l=0}^j\frac{j!}{l!}(-1)^{j-l}\beta^{l-j-1}\partial_\beta^lF_1.
\end{equation*}
Then, by \eqref{eq 1 potenza laplaciano sferica} we have
\begin{equation*}
\begin{split}
    \Delta_{m+1}^kf'_s&=(m-3)\cdot...\cdot(m-2k-1)\sum_{j=1}^{k}a_j^{(k)}\beta^{j-2k}\partial_\beta^j(\beta^{-1}F_1)= \\
    &=(m-3)\cdot...\cdot(m-2k-1)\sum_{j=1}^{k}\sum_{l=0}^ja_j^{(k)}\beta^{j-2k}\frac{j!}{l!}(-1)^{j-l}\beta^{l-j-1}\partial_\beta^lF_1\\
    &=(m-3)\cdot...\cdot(m-2k-1)\sum_{l=0}^{k}\sum_{j=l}^ka_j^{(k)}\beta^{l-2k-1}\frac{j!}{l!}(-1)^{j-l}\partial_\beta^lF_1\\
    &=(m-3)\cdot...\cdot(m-2k-1)\sum_{j=0}^{k}\beta^{j-2k-1}\sum_{l=j}^ka_l^{(k)}\frac{l!}{j!}(-1)^{l-j}\partial_\beta^jF_1\\
    &=(m-3)\cdot...\cdot(m-2k-1)\sum_{j=0}^{k}a_{j+1}^{(k+1)}\beta^{j-2k-1}\partial_\beta^jF_1\\
    &=(m-3)\cdot...\cdot(m-2k-1)\sum_{j=1}^{k+1}a_{j}^{(k+1)}\beta^{j-2k-2}\partial_\beta^{j-1}F_1,
\end{split}
\end{equation*}
    where we have applied \ref{lem:coefficienti}.
\end{proof}

We have similar formulas for $\Delta_{m+1}^kf$ applied to a slice regular function $f$. We first state the case with $k=1$.
\begin{lem}\label{lem:laplaciano slice regular}
    Let let $f\in\mathcal{SR}(\Omega_D)$, then
        $\Delta_{m+1}f=2(1-m)\frac{\partial}{\partial x}f'_s$.
\end{lem}
\begin{proof}
    Let $f=\mathcal{I}(F)$, with $F(\alpha+i\beta)=F_0(\alpha+i\beta)+e_1F_1(\alpha+i\beta)=F_0(\alpha+i\beta)+e_1\beta^{-1}F'_s(\alpha+i\beta)$ and $\Delta_2F_0=\Delta_2F_1=0$, by the slice regularity of $f$. Then, for any $i=1,\dots,n$ we have
    \begin{equation*}
        \partial_{x_i}f(\alpha+J\beta)=x_i\beta^{-1}\partial_\beta F_0(\alpha+i\beta)+e_i F'_s(\alpha+i\beta)+J\beta(-x_i\beta^{-3}F_1(\alpha+i\beta)+x_i\beta^{-2}\partial_\beta F_1(\alpha+i\beta))
    \end{equation*}
    and
    \begin{align*}
        \partial_{x_i}^2f(\alpha+J\beta)&=\beta^{-1}\partial_\beta F_0(\alpha+i\beta)-x_i^2\beta^{-3}\partial_\beta F_0(\alpha+i\beta)+x_i^2\beta^{-2}\partial_\beta^2 F_0(\alpha+i\beta)+\\
        &-2x_ie_i\beta^{-3}F_1(\alpha+i\beta)+2x_ie_i\beta^{-2}\partial_\beta F_1(\alpha+i\beta)-J\beta^{-2}F_1+\\
        &+3x_i^2J\beta^{-4}F_1(\alpha+i\beta)-x_i^2J\beta^{-3}\partial_\beta F_1(\alpha+i\beta)+J\beta^{-1}\partial_\beta F_1(\alpha+i\beta)+\\
        &-2x_i^2J\beta^{-3}\partial_\beta F_1(\alpha+i\beta)+x_i^2\beta^{-2}J\partial_\beta^2F_1(\alpha+i\beta).
    \end{align*}
    Then
\begin{align*}
    \Delta_{m+1}f(\alpha+J\beta)&=\partial_{x_0}^2f(\alpha+J\beta)+\textstyle\sum_{i=1}^m\partial_{x_i}^2f(\alpha+J\beta)\\
    &=\partial_{\alpha}^2F_0(\alpha+i\beta)+J\partial_{\alpha}^2F_1(\alpha+i\beta)+m\beta^{-1}\partial_\beta F_0(\alpha+i\beta)+\\
    &-\beta^{-1}\partial_\beta F_0(\alpha+i\beta)+\partial_\beta^2 F_0(\alpha+i\beta)-2J\beta^{-2}F_1(\alpha+i\beta)+2J\beta^{-1}\partial_\beta F_1(\alpha+i\beta)+\\
    &-mJ\beta^{-2}F_1+3J\beta^{-2}F_1(\alpha+i\beta)-J\beta^{-1}\partial_\beta F_1(\alpha+i\beta)+mJ\beta^{-1}\partial_\beta F_1(\alpha+i\beta)+\\
    &-2J\beta^{-1}\partial_\beta F_1(\alpha+i\beta)+J\partial_\beta^2F_1(\alpha+i\beta)\\
    &=(m-1)[\beta^{-1}\partial_\beta F_0(\alpha+i\beta)-J\beta^{-2}F_1+J\beta^{-1}\partial_\beta F_1(\alpha+i\beta)].
\end{align*}
On the other hand,
\begin{align*}
    \frac{\partial}{\partial x}f'_s(\alpha+J\beta)&=\mathcal{I}\left(\frac{1}{2}(\partial_\alpha-e_1\partial\beta)[\beta^{-1}F_1(\alpha+i\beta)]\right)(\alpha+J\beta)\\
    &=\mathcal{I}\left(\frac{1}{2}\beta^{-1}\partial_\alpha F_1+\frac{1}{2}e_1\beta^{-2}F_1-\frac{1}{2}e_1\beta^{-1}\partial_\beta F_1\right)(\alpha+J\beta)\\
    &=\mathcal{I}\left(-\frac{1}{2}\beta^{-1}\partial_\beta F_0+\frac{1}{2}e_1\beta^{-2}F_1-\frac{1}{2}e_1\beta^{-1}\partial_\beta F_1\right)(\alpha+J\beta)\\
    &=\frac{1}{2}\left(-\beta^{-1}\partial_\beta F_0(\alpha+i\beta)+J\beta^{-2}F_1(\alpha+i\beta)-J\beta^{-1}\partial_\beta F_1(\alpha+i\beta)\right).
\end{align*}
\end{proof}
\begin{oss}
    Despite this result is proven in \cite{perotti2022cauchy}, we showed a proof that require only the holomorphicity of $F$, in order to obtain an analogous result in several variables. More generally, Lemma \ref{lem:laplaciano slice regular} holds for any $af$ with $a\in\mathbb{R}_m$, $f\in\mathcal{S}(\Omega_D)$ and $\partial/\partial x^cf=0$.
\end{oss}

In order to give explicit formulas for the iterated Laplacian applied to slice regular functions, we need to switch the Laplacian and the slice derivative. Next Lemma shows that they commute when applied to circular slice functions.
\begin{lem}
\label{lemma Laplaciano e derivata slice commutano su funzioni circolari}
    Let $f\in\mathcal{S}_c(\Omega_D)\cap C^3(\Omega_D)$, then it holds
    \begin{equation}
    \label{eq. Laplaciano e derivata slice commutano su funzioni circolari}
\Delta_{m+1}\left(\frac{\partial f}{\partial x}\right)=\frac{\partial}{\partial x}(\Delta_{m+1} f).
    \end{equation}
\end{lem}
\begin{proof}
    Since $f\in\mathcal{S}_c(\Omega_D)$, $f(x)=F(x_0,\beta(x_1,\dots, x_m))$, where $\beta(x_1,\dots,x_m)=\sqrt{x_1^2+\dots x_m^2}$. By definition, if $x=\alpha+J\beta\in\mathbb{R}^{m+1}$,
    \begin{equation*}
\frac{\partial f}{\partial x}(x)=\frac{1}{2}\left(\partial_\alpha-J\partial_\beta\right)F(x_0,\beta)=\frac{1}{2}\left(\partial_\alpha F(x_0,\beta)-J\partial_\beta F(x_0,\beta)\right).
    \end{equation*}
    Note that for any $i=1,\dots m$, $\partial_{x_i}J=e_i\beta^{-1}-x_iJ\beta^{-2}$, so
    \begin{equation*}
\partial_{x_i}\left(\partial_\alpha F-J\partial_\beta F\right)=x_i\beta^{-1}\partial_\beta \partial_\alpha F-e_i\beta^{-1}\partial_\beta F+x_iJ\beta^{-2}\partial_\beta F-x_iJ\beta^{-1}\partial_\beta^2 F
    \end{equation*}
    and 
    \begin{equation*}
\begin{split}
    \partial_{x_i}^2\left(\partial_\alpha F-J\partial_\beta F\right)&=\beta^{-1}\partial_\beta\partial_\alpha F-x_i^2\beta^{-3}\partial_\beta\partial_{\alpha} F+x_i^2\beta^{-2}\partial_\beta^2\partial_\alpha F+x_ie_i\beta^{-3}\partial_\beta F-x_ie_i\beta^{-2}\partial_\beta^2 F+\\
    &+J\beta^{-2}\partial_\beta F+x_ie_i\beta^{-3}\partial_\beta F-x_i^2J\beta^{-4}\partial_\beta F-2x_i^2J\beta^{-4}\partial_\beta F+x_i^2J\beta^{-3}\partial_\beta^2 F+\\
    &-x_ie_i\beta^{-2}\partial_\beta^2 F+x_i^2J\beta^{-3}\partial^2_\beta F-J\beta^{-1}\partial_\beta^2F+x_i^2J\beta^{-3}\partial_\beta^2F-x_i^2J\beta^{-2}\partial_\beta^3F.
\end{split}
    \end{equation*}
    Thus, if $\Delta_m=\sum_{i=1}^m\partial_{x_i}^2$, we have
    \begin{equation*}
\begin{split}
    \Delta_m(\partial_\alpha F-J\partial_\beta F)&=m\beta^{-1}\partial_\beta\partial_\alpha F-\beta^{-1}\partial_\beta\partial_\alpha F+\partial_\beta^2\partial_\alpha F+J\beta^{-2}\partial_\beta F-J\beta^{-1}\partial_\beta^2 F+\\
    &+mJ\beta^{-2}\partial_\beta F+J\beta^{-2}\partial_\beta F-J\beta^{-2}\partial_\beta F-2J\beta^{-2}\partial_\beta F+J\beta^{-1}\partial_\beta^2 F+\\
    &-mJ\beta^{-1}\partial_\beta^2 F+J\beta^{-1}\partial_\beta^2F-J\partial_\beta^3F\\
    &=(m-1)\left[\beta^{-1}\partial_\beta\partial_\alpha F+J\beta^{-2}\partial_\beta F-J\beta^{-1}\partial_\beta^2 F\right]+\partial_\beta^2\partial_\alpha F-J\partial_\beta^3 F.
\end{split}
    \end{equation*}
    So, the left hand side of \eqref{eq. Laplaciano e derivata slice commutano su funzioni circolari} becomes
    \begin{equation*}
    \begin{split}
\Delta_{m+1}\left(\frac{\partial f}{\partial x}\right)=&\frac{1}{2}\left[\partial_\alpha^3 F+\partial_\beta^2\partial_\alpha F-J\partial_\beta\partial_\alpha^2 F-J\partial_\beta^3 F+\right.\\
&\left.+(m-1)(\beta^{-1}\partial_\beta\partial_\alpha F+J\beta^{-2}\partial_\beta F-J\beta^{-1}\partial_\beta^2 F)\right].
\end{split}
    \end{equation*}
    On the other hand, $\partial_{x_i}f=x_i\beta^{-1}\partial_\beta F$ and $\partial_{x_i}^2f=\beta^{-1}\partial_\beta F-x_i^2\beta^{-3}\partial_\beta F+x_i^2\beta^{-2}\partial_\beta^2F$,    so
    \begin{equation*}
\Delta_{m+1}f=\partial_\alpha^2F+(m-1)\beta^{-1}\partial_\beta F+\partial_\beta^2F
    \end{equation*}
    and finally the right hand side of \eqref{eq. Laplaciano e derivata slice commutano su funzioni circolari} becomes
    \begin{equation*}
\begin{split}
    \frac{\partial}{\partial x}(\Delta_{m+1}f)&=\frac{1}{2}(\partial_\alpha-J\partial_\beta)(\partial_\alpha^2F+(m-1)\beta^{-1}\partial_\beta F+\partial_\beta^2F)\\
    &=\frac{1}{2}\left[\partial_\alpha^3 F+\partial_\beta^2\partial_\alpha F-J\partial_\beta\partial_\alpha^2 F-J\partial_\beta^3 F+\right.\\
    &\quad \left.+(m-1)(\beta^{-1}\partial_\beta\partial_\alpha F+J\beta^{-2}\partial_\beta F-J\beta^{-1}\partial_\beta^2 F)\right].
\end{split}
    \end{equation*}
This proves \eqref{eq. Laplaciano e derivata slice commutano su funzioni circolari}.
    \end{proof}

    \begin{thm}
\label{theorem laplacian any order slice regular functions}
    Let $f\in\mathcal{S}\mathcal{R}(\Omega_D)$, then for any $k=1,2,\dots$ it holds
    \begin{equation}
    \label{eq formula laplaciano iterato 1}
\Delta_{m+1}^{k+1}f=-2(m-1)\cdot\dots\cdot(m-2k-1)\sum_{j=1}^ka_j^{(k)}\frac{\partial}{\partial x}\left(\beta^{j-2k}\partial_\beta^jf'_s\right),
    \end{equation}
    or equivalently, if $f=\mathcal{I}(F_0+e_1F_1)$,
    \begin{equation}
\label{eq formula laplaciano iterato 2}
\Delta_{m+1}^{k+1}f=-2(m-1)\cdot\dots\cdot(m-2k-1)\sum_{j=1}^{k+1}a_j^{(k+1)}\frac{\partial}{\partial x}\left(\beta^{j-2k-2}\partial_\beta^{j-1}F_1\right).
    \end{equation}
     In particular, $f$ is $(\gamma_m+1)$-polyharmonic, i.e.
$\Delta^{\frac{m+1}{2}}_{m+1}f=0$.
\end{thm}

\begin{proof}
The case $k=0$ is proven in Lemma \ref{lem:laplaciano slice regular}. Assume $k>1$, and let us prove \eqref{eq formula laplaciano iterato 1}. By \eqref{eq 1 potenza laplaciano sferica}, if $f\in\mathcal{S}\mathcal{R}(\Omega_D)$ we have
    \begin{equation*}
\Delta_{m+1}^{k}f'_s=(m-3)\cdot\dots\cdot(m-2k-1)\sum_{j=1}^ka_j^{(k)}\beta^{j-2k}\partial_\beta^jf'_s.
    \end{equation*}
    Now, using Lemma \ref{lem:laplaciano slice regular} and Lemma \ref{lemma Laplaciano e derivata slice commutano su funzioni circolari}, which applies, since $f'_s$ is circular, we have
    \begin{equation*}
\begin{split}
    \Delta_{m+1}^{k+1}f&=\Delta^k_{m+1}(\Delta_{m+1}f)=-2(m-1)\Delta_{m+1}^k\left(\frac{\partial f'_s}{\partial x}\right)=-2(m-1)\frac{\partial}{\partial x}\left(\Delta_{m+1}^kf'_s\right)\\
    &=-2(m-1)(m-3)\cdot\dots\cdot(m-2k-1)\sum_{j=1}^ka_j^{(k)}\frac{\partial}{\partial x}\left(\beta^{j-2k}\partial_\beta^jf'_s\right).
\end{split}
    \end{equation*}
Finally, \eqref{eq formula laplaciano iterato 2} follows analogously by using \eqref{eq 2 potenza laplaciano sferica}, instead of \eqref{eq 1 potenza laplaciano sferica}.
\end{proof}

\subsection{Polyharmonicity of several Clifford variables slice regular functions}

In this subsection we assume $\Omega_D\subset(\mathbb{R}^{m+1})^n$ open circular set. The next result extends Proposition \ref{Prop potenza laplaciano derivata sferica} to several variables.
    \begin{prop}\label{Prop:iteratedLaplaciansevvar}
        Let $f=\mathcal{I}(F)\in\mathcal{S}(\Omega_D)\cap\ker(\partial/\partial x_h^c)$, for some $h=1,\dots,n$, with $F=\sum_Ke_KF_K$. Then, for any $k=1,2,\dots$, it holds
    \begin{equation}\label{eq:Laplacianiderivatasfericasevvar1}
        \Delta^k_{m+1,h}f'_{s,h}(x)=(m-3)\cdots(m-2k-1)\sum_{j=1}^ka_j^{(k)}\beta_h^{j-2k}\sum_{K\in\mathcal{P}(n)\setminus\{h\}}[J_K,\partial_{\beta_h}^j(F'_h)_K(w',\alpha_h+i\beta_h,w'')]
    \end{equation}
    or, equivalently,
    \begin{equation}\label{eq:Laplacianiderivatasfericasevvar2}
        \Delta^k_{m+1,h}f'_{s,h}(x)=(m-3)\cdots(m-2k-1)\sum_{j=1}^{k+1}a_j^{(k+1)}\beta_h^{j-2k-2}\sum_{K\in\mathcal{P}(n)\setminus\{h\}}[J_K,\partial_{\beta_h}^jF_{K\cup\{h\}}(w',\alpha_h+i\beta_h,w'')].
    \end{equation}
    In particular, $f'_{s,h}$ is $\gamma_m$-polyharmonic in $x_h$, i.e. $\Delta_{m+1,h}^{\gamma_m}f'_{s,h}=0$.
\end{prop}
In order to prove Proposition \ref{Prop:iteratedLaplaciansevvar}, we will proceed as in the one variable case. First, we need a several variables version of \cite[Theorem 4.1]{Harmonicity}.
Let $h\in\{1,\dots,n\}$ and let $F=\sum_{K\in\mathcal{P}(n)}e_KF_K:D\subset\mathbb{C}^n\to\mathbb{R}_m\otimes\mathbb{R}^{2^n}$ be a stem function, whose components are analytic functions with respect to $z_h$. Since, by definition, for any $K\in\mathcal{P}(n)\setminus\{h\}$, $F_{K\cup\{h\}}$ is odd with respect to $\beta_h$, by Whitney's Theorem \cite{Whitney}, there exist functions $G_K:D\to\mathbb{R}_m$, analytic with respect to $z_h$, such that $$\beta_hG_K(z_1,\dots,z_{h-1},\alpha_h+i\beta^2_h,z_{h+1},\dots,z_n)=F_{K\cup\{h\}}(z_1,\dots,z_{h-1},\alpha_h+i\beta_h,z_{h+1},\dots,z_n).$$
\begin{lem}\label{lem:Ginsevevar}
    Let $\Omega_D\subset(\mathbb{R}^{m+1})^n$ and let $f\in\mathcal{S}(\Omega_D)\cap\ker(\partial/\partial x_h^c)$, for some $h=1,\dots,n$. Then, for any $k=1,2,\dots$, it holds
    \begin{equation*}
        \Delta^k_{m+1,h}f'_{s,h}(x)=2^k(m-3)\cdots(m-2k-1)\sum_{K\in\mathcal{P}(n), h\notin K}[J_K,\partial^k_{2,h}G_K(w',\alpha_h+i\beta_h^2,w'')],
    \end{equation*}
    for any $x=(\alpha_1+J_1\beta_1,\dots,\alpha_n+J_n\beta_n)$, where $w'=(\alpha_1+i\beta_1,\dots,\alpha_{h-1}+i\beta_{h-1})$ and $w''=(\alpha_{h+1}+i\beta_{h+1},\dots,\alpha_{n}+i\beta_{n})$. 
    \end{lem}

    \begin{proof}
    We have $\Delta^k_{m+1,h}f'_{s,h}(x)=\Delta^k_{m+1}(f^{(x)}_h)'_s(x_h)$, with $f^{(x)}_h=\mathcal{I}(F^{(x)}_h)$, where if $x=(\alpha_1+J_1\beta_1,\dots,\alpha_n+J_n\beta_n)$, $w'=(\alpha_1+i\beta_1,\dots,\alpha_{h-1}+i\beta_{h-1})$ and $w''=(\alpha_{h+1}+i\beta_{h+1},\dots,\alpha_{n}+i\beta_{n})$, then $F^{(x)}_h=F^{(x)}_{0,h}+e_1F^{(x)}_{1,h}$, with $$F^{(x)}_{0,h}(z)=\sum_{K\in\mathcal{P}(n),h\notin K}[J_K,F_K(w',z,w'')],\qquad F^{(x)}_{1,h}(z)=\sum_{K\in\mathcal{P}(n),h\notin K}[J_K,F_{K\cup\{h\}}(w',z,w'')].$$
    Since for any $x\in\Omega_D$, $F^{(x)}_{1,h}$ is an analytic odd function w.r.t. $\beta$, by Whitney's Theorem there exist a real analytic function $G^{(x)}_h$ such that $\beta G^{(x)}_h(\alpha+i\beta^2)=F^{(x)}_{1,h}(\alpha+i\beta)$ and since $\Delta_2F^{(x)}_{1,h}=0$, by \cite{Harmonicity} we have that
    \begin{equation}\label{eq:laplacianopreliminare}
        \Delta^k_{m+1,h}f'_{s,h}(x)=2^k(m-3)\cdots(m-2k-1)\partial_2^kG^{(x)}_h(\operatorname{Re}(x_h),|\operatorname{Im}(x_h)|^2).
    \end{equation}
    On the other hand, for every $K\in\mathcal{P}(n)$,  we have that $\beta_hG_K(w',\alpha_h+i\beta_h^2,w'')=F_{K\cup\{h\}}(w',\alpha_h+i\beta_h,w'')$. Thus,
    \begin{align*}
        \beta_hG^{(x)}_h(\alpha_h+i\beta_h^2)&=F^{(x)}_{1,h}(\alpha_h+i\beta_h)=\sum_{K\in\mathcal{P}(n),h\notin K}[J_K,F_{K\cup\{h\}}(w',\alpha_h+i\beta_h,w'')]\\
        &=\sum_{K\in\mathcal{P}(n),h\notin K}[J_K,\beta_hG_K(w',\alpha_h+i\beta_h^2,w'')],
    \end{align*}
    from which we get $G^{(x)}_h(\alpha_h+i\beta_h^2)=\sum_{K\in\mathcal{P}(n),h\notin K}[J_K,G_K(w',\alpha_h+i\beta_h^2,w'')]$ and by \eqref{eq:laplacianopreliminare} we conclude. 
\end{proof}

Now, it is easy to prove Proposition \ref{Prop:iteratedLaplaciansevvar}.
\begin{proof}[Proof of Proposition \ref{Prop:iteratedLaplaciansevvar}]
    Formula \eqref{eq:Laplacianiderivatasfericasevvar1} follows from Lemma \ref{lem:Ginsevevar} and \cite[Proposition 4.2]{SCHCF}, while \eqref{eq:Laplacianiderivatasfericasevvar2} by \eqref{eq:Laplacianiderivatasfericasevvar1} and  Lemma \ref{lem:coefficienti}.
\end{proof}

We show that $\Delta_{m+1,h}$ and $\partial/\partial x_h$ commute when applied to functions in $\mathcal{S}_{c,h}(\Omega_D)$.
\begin{lem}
    Let $f\in\mathcal{S}_{c,h}(\Omega_D)\cap C^3(\Omega_D)$, for some $h=1,\dots,n$. Then it holds
    \begin{equation}
    \label{eq. h-Laplaciano e derivata slice commutano su funzioni circolari}
\Delta_{m+1,h}\left(\frac{\partial f}{\partial x_h}\right)=\frac{\partial}{\partial x_h}(\Delta_{m+1,h} f).
    \end{equation}
\end{lem}
\begin{proof}
    By Proposition \ref{proposizione proprieta derivata sferica} (1), $f'_{s,h}\in\mathcal{S}_{c,h}(\Omega_{D_h})$, then for any fixed $y\in\Omega_{D_h}$, $f^{(y)}_h\in\mathcal{S}_c(\Omega_{D,h})$, hence by \eqref{eq. Laplaciano e derivata slice commutano su funzioni circolari} we have
    \begin{equation*}
        \Delta_{m+1,h}\left(\frac{\partial f}{\partial x_h}(y)\right)=\Delta_{m+1}\frac{\partial}{\partial x}(f^{(y)}_h(x))=\frac{\partial}{\partial x}\left(\Delta_{m+1}f^{(y)}_h(x)\right)=\frac{\partial}{\partial x_h}(\Delta_{m+1,h} f(y)).
    \end{equation*}
\end{proof}

\begin{thm}\label{thm:iteratedlaplaciansevvar}
    Let $f=\mathcal{I}(F)\in\mathcal{S}(\Omega_D)\cap\ker(\partial/\partial x_h^c)$, for some $h=1,\dots,n$, with $F=\sum_Ke_KF_K$. Then, for any $k\in\mathbb{N}$, it holds
    \begin{equation}\label{eq:Laplaciano slice reg sev var 1}
        \Delta_{m+1,h}^{k+1}f(x)=2(m-1)\cdots(m-2k-1)\sum_{j=1}^ka_j^{(k)}\frac{\partial}{\partial x_h}\beta_h^{j-2k}\sum_{K\in\mathcal{P}(n)\setminus\{h\}}[J_K,\partial_{\beta_h}^j(F'_h)_K(w',\alpha_h+i\beta_h,w'')],
    \end{equation}
    or equivalently
    \begin{equation}\label{eq:Laplaciano slice reg sev var 2}
        \Delta_{m+1,h}^{k+1}f(x)=2(m-1)\cdots(m-2k-1)\sum_{j=1}^{k+1}a_j^{(k+1)}\frac{\partial}{\partial x_h}\beta_h^{j-2k-2}\sum_{K\in\mathcal{P}(n)\setminus\{h\}}[J_K,\partial_{\beta_h}^jF_{K\cup\{h\}}(w',\alpha_h+i\beta_h,w'')],
    \end{equation}
    where for $k=0$, we mean 
    \begin{equation}\label{eq:Leplaciano 1 sev var}
        \Delta_{m+1,h}f(x)=-2(m-1)\frac{\partial}{\partial x_h}f'_{s,h}(x).
    \end{equation}
    In particular, $f$ is $(\gamma_m+1)$-polyharmonic in any variable, i.e. $\Delta_{m+1,h}^\frac{m+1}{2}f=0$, for any $h=1,\dots, n$.
\end{thm}

\begin{proof}
First, let us prove \eqref{eq:Leplaciano 1 sev var}. For any fixed $y=(y_1,\dots,y_n)\in\Omega_D$, with $y_\ell=\alpha_l+i\beta_\ell$, set $w'=(\alpha_1+i\beta_1,\dots,\alpha_{h-1}+i\beta_{h-1})$ and $w''=(\alpha_{h+1}+i\beta_{h+1},\dots,\alpha_n+i\beta_n)$, then for any $x=\alpha+J\beta=\phi_J(z)$ we have
\begin{align*}
    f^{(y)}_h(x)&=\sum_{K\in\mathcal{P}(n)}[J_K,F_K(w',z,w'')]\\
    &=\sum_{\substack{P\in\mathcal{P}(h-1) \\ Q\subset\{h+1,\dots,n\}}}[J_{P\cup Q},F_{P\cup Q}(w',z,w'')]+[J_{P\cup\{h\}\cup Q},F_{P\cup\{h\}\cup Q}(w',z,w'')]\\
    &=\sum_{\substack{P=(p_1,\dots,p_s)\in\mathcal{P}(h-1) \\ Q=(q_1,\dots,q_t)\subset\{h+1,\dots,n\}}}J_{p_1}\dots J_{p_s}\left(J_{q_1}\dots J_{q_t}F_{P\cup Q}(w',z,w'')+JJ_{q_1}\dots J_{q_t}F_{P\cup\{h\}\cup Q}(w',z,w'')\right).
\end{align*}
Set $$g^{(y)}(\alpha+J\beta)=\sum_{\substack{P\in\mathcal{P}(h-1) \\ Q\subset\{h+1,\dots,n\}}}J_{q_1}\dots J_{q_t}F_{P\cup Q}(w',\alpha+i\beta,w'')+JJ_{q_1}\dots J_{q_t}F_{P\cup\{h\}\cup Q}(w',\alpha+i\beta,w'').$$ Since $f\in\ker(\partial/\partial x_h^c)$, $g$ is a slice regular function, so, by Lemma \ref{lem:laplaciano slice regular}, it holds
\begin{align*}
    \Delta_{m+1,h}f(y)&=\Delta_{m+1}(f^{(y)}_h)(x)=\sum_{P\in\mathcal{P}(h-1)}J_P\Delta_{m+1}g^{(y)}(x)=2(1-m)\sum_{P\in\mathcal{P}(h-1)}J_P\frac{\partial}{\partial x}(g^{(y)})'_s(x)\\
    &=2(1-m)\frac{\partial}{\partial x}\sum_{P\in\mathcal{P}(h-1)}J_P(g^{(y)})'_s(x)=2(1-m)\frac{\partial}{\partial x}(f^{(y)})'_s(x)=2(1-m)\frac{\partial}{\partial x_h}f'_{s,h}(x).
\end{align*}
    Equation \eqref{eq:Laplaciano slice reg sev var 1} follows from \eqref{eq. h-Laplaciano e derivata slice commutano su funzioni circolari} and \eqref{eq:Laplacianiderivatasfericasevvar1}, while \eqref{eq:Laplaciano slice reg sev var 2} follows from \eqref{eq. h-Laplaciano e derivata slice commutano su funzioni circolari} and \eqref{eq:Laplacianiderivatasfericasevvar2}.
\end{proof}

We can actually prove something more, namely a method to construct polyharmonic functions, starting by harmonic functions in the plane.
\begin{prop}
\label{prop method for constructing polyharmonic functions}
    Let $m$ be odd and let $F:D\subset\mathbb{R}\times\mathbb{R}^+\to \mathbb{R}_m$ be harmonic function, i.e. $\Delta_2F(a,b)=(\partial^2_a+\partial^2_b)F(a,b)=0$ and let $f:\Omega_D\subset\mathbb{R}^{m+1}\to\mathbb{R}_m$, $$f(a,x_1,\dots,x_m)\coloneqq \frac{1}{\sqrt{x_1^2+\dots x_m^2}}F\left(a,\sqrt{x_1^2+\dots x_m^2}\right).$$ Then, for any $k\in\mathbb{N}$ it holds
$$
\Delta_{m+1}^kf(a,x_1,\dots,x_m)=(m-3)\cdot...\cdot(m-2k-1)\sum_{j=1}^{k+1}a_j^{(k+1)}b^{j-2k-2}\partial_b^{j-1}F(a,b(x_1,\dots,x_m)),
$$
where $\Delta_{m+1}=\partial^2_a+\sum_{j=1}^m\partial^2_{x_j}$ is the Laplacian of $\mathbb{R}^{m+1}$.
In particular, $f$ is polyharmonic of degree ${\gamma_m}$, i.e.
$$\Delta^{{\gamma_m}}_{m+1}f=0.$$
\end{prop}
\begin{proof}
Note that for $k=0$ we get $f=\frac{1}{\sqrt{x_1^2+\dots x_m^2}}F$. Now, suppose by induction that for some $k$ it holds
$$
\Delta_{m+1}^{k-1}f=(m-3)\cdot...\cdot(m-2k+1)\sum_{j=1}^ka_j^{(k)}b^{j-2k}\partial_b^{j-1}F
$$
then let us compute $\partial^2_{x_i}(b^{j-2k}\partial_b^{j-1}F)$ for $i=1,...,n$. Recalling that 
$$
\partial_{x_i}(b^n)=nx_ib^{n-2},\qquad \partial_{x_i}F=x_ib^{-1}\partial_bF,
$$
we have
$$
\partial_{x_i}(b^{j-2k}\partial_b^{j-1}F)=(j-2k)x_ib^{j-2k-2}\partial_b^{j-1}F+x_ib^{j-2k-1}\partial_b^jF,
$$
and 
\begin{equation*}
    \begin{split}
\partial^2_{x_i}(b^{j-2k}\partial_b^{j-1}F)&=(j-2k)b^{j-2k-2}\partial_b^{j-1} F+(j-2k)(j-2k-2)x_i^2b^{j-2k-4}\partial_b^{j-1} F+\\
&+(j-2k)x_i^2b^{j-2k-3}\partial_b^j F+b^{j-2k-1}\partial_b^j F+\\
&+(j-2k-1)x_i^2b^{j-2k-3}\partial_b^j F+x_i^2b^{j-2k-2}\partial_b^{j+1}F.
\end{split}
\end{equation*}
Now, since $b^2=\sum_{i=1}^mx_i^2$
\begin{equation*}
    \begin{split}
&\sum_{i=1}^m\partial^2_{x_i}(b^{j-2k}\partial_b^{j-1} F)=m(j-2k)b^{j-2k-2}\partial_b^{j-1} F+(j-2k)(j-2k-2)b^{j-2k-2}\partial_b^{j-1} F+\\
&+(j-2k)b^{j-2k-1}\partial_b^j F+mb^{j-2k-1}\partial_b^j F+\\
&+(j-2k-1)b^{j-2k-1}\partial_b^j F+b^{j-2k}\partial_b^{j+1} F\\
&=(m+j-2k-2)(j-2k)b^{j-2k-2}\partial_b^{j-1} F+(m+2j-4k-1)b^{j-2k-1}\partial_b^j F+\\
&+b^{j-2k}\partial_b^{j+1} F
\end{split}
\end{equation*}
and by the harmonicity of $F$,
\begin{equation*}
    \begin{split}
&\Delta_{m+1}(b^{j-2k}\partial_b^{j-1} F)=\left(\partial_a^2+\sum_{i=1}^m\partial^2_{x_i}\right)(b^{j-2k}\partial_b^{j-1} F)\\
&=(m+j-2k-2)(j-2k)b^{j-2k-2}\partial_b^{j-1} F+(m+2j-4k-1)b^{j-2k-1}\partial_b^j F
+\\
&\ +b^{j-2k}\partial_b^{j-1} (\partial_b^2F+\partial_a^2F)\\
&=(m+j-2k-2)(j-2k)b^{j-2k-2}\partial_b^{j-1} F+(m+2j-4k-1)b^{j-2k-1}\partial_b^j F.
\end{split}
\end{equation*}
Let us split
$m+j-2k-2=m-2k-1+j-1$ and $m+2j-4k-1=m-2k-1+2j-2k$, so we have
\begin{equation*}
    \begin{split}
\Delta_{m+1}(b^{j-2k}\partial_b^{j-1} F)&=(m-2k-1)[(j-2k)b^{j-2k-2}\partial_b^{j-1} F+b^{j-2k-1}\partial_b^j F]+\\
&+(j-1)(j-2k)b^{j-2k-2}\partial_b^{j-1} F+2(j-k)b^{j-2k-1}\partial_b^j F
\end{split}
\end{equation*}
and considering the whole function $f$ we have
\begin{equation*}
    \begin{split}
\Delta_{m+1}^{k}f&=\Delta_{m+1}(\Delta_{m+1}^{k-1}f)=(m-3)\cdot...\cdot(m-2k+1)\sum_{j=1}^ka_j^{(k)}\Delta_{m+1}(b^{j-2k}\partial_b^{j-1} F)\\
&=(m-3)\cdot...\cdot(m-2k+1)(m-2k-1)\sum_{j=1}^ka_j^{(k)}[(j-2k)b^{j-2k-2}\partial_b^{j-1} F+b^{j-2k-1}\partial_b^j F]+\\
&+(m-3)\cdot...\cdot(m-2k+1)\sum_{j=1}^ka_j^{(k)}[(j-1)(j-2k)b^{j-2k-2}\partial_b^{j-1} F+2(j-k)b^{j-2k-1}\partial_b^j F].
\end{split}
\end{equation*}
Let us focus on the second sum and let us prove that it is actually zero. Indeed, we have
\begin{equation*}
    \begin{split}
&\sum_{j=1}^ka_j^{(k)}(j-1)(j-2k)b^{j-2k-2}\partial_b^{j-1} F+\sum_{j=1}^ka_j^{(k)}2(j-k)b^{j-2k-1}\partial_b^j F\\
&=\sum_{j=2}^ka_j^{(k)}(j-1)(j-2k)b^{j-2k-2}\partial_b^{j-1} F+\sum_{j=2}^ka_{j-1}^{(k)}2(j-k-1)b^{j-2k-2}\partial_b^{j-1} F\\
&=\sum_{j=2}^k[a_{j-1}^{(k)}2(j-k-1)+a_{j-1}^{(k)}2(j-k-1)]b^{j-2k-2}\partial_b^{j-1} F,
    \end{split}
\end{equation*}
but, by definition of $a_j^{(k)}$
\begin{equation*}
    \begin{split}
&a_{j-1}^{(k)}2(j-k-1)+a_{j-1}^{(k)}2(j-k-1)\\
&=\frac{(2k-j-1)!}{(j-1)!(k-j)!(-2)^{k-j}}(j-1)(j-2k)+\frac{(2k-j)!}{(j-2)!(k-j+1)!(-2)^{k-j+1}}2(j-k-1)\\
&=\frac{-(2k-j)!}{(j-2)!(k-j)!(-2)^{k-j}}+\frac{(2k-j)!}{(j-2)!(k-j)!(-2)^{k-j}}=0.
    \end{split}
\end{equation*}
So, finally
\begin{equation*}
    \begin{split}
    \Delta_{m+1}^{k}f
&=(m-3)\cdot...\cdot(m-2k+1)(m-2k-1)\sum_{j=1}^ka_j^{(k)}[(j-2k)b^{j-2k-2}\partial_b^{j-1} F+b^{j-2k-1}\partial_b^j F]\\
&=(m-3)\cdot...\cdot(m-2k-1)\left[\sum_{j=1}^{k+1}a_{j}^{(k)}(j-2k)b^{j-2k-2}\partial_b^{j-1} F+\sum_{j=1}^{k+1}a_{j-1}^{(k)}b^{j-2k-2}\partial_b^{j-1} F\right]\\
&=(m-3)\cdot...\cdot(m-2k-1)\sum_{j=1}^ka_j^{(k+1)}b^{j-2k-2}\partial_b^{j-1} F,
\end{split}
\end{equation*}
where we have used the property $a_j^{(k+1)}=a_{j-1}^{(k)}+(j-2k)a_j^{(k)}$ and that $a_j^{(k)}=0$ if $j\notin\{1,\dots, k\}$.
\end{proof}
\begin{oss}
    Note that, $f$ may not be a slice function. Indeed, consider $F(a,b)=a^4-6a^2b^2+b^4$, then 
    \begin{equation*}
f(x_0,x_1,x_2,x_3)=\frac{x_0^4}{\sqrt{x_1^2+x_2^2+x_3^2}}-6x_0^2\sqrt{x_1^2+x_2^2+x_3^2}+(x_1^2+x_2^2+x_3^2)^\frac{3}{2}=\frac{(x^4)^\circ_s}{\sqrt{x_1^2+x_2^2+x_3^2}}
    \end{equation*}is an harmonic function, which is not slice. This follows from the unicity of the stem function and that $f$ would be induced by $F$, which is not a Stem function, since it does not satisfy \eqref{eq defn stem functions}.
\end{oss}
\section{Almansi decomposition}
\subsection{Classical Almansi decomposition}
We recall the Classical Almansi Theorem \cite{AlmansiClassico}. 
\begin{thm}
\label{Classical Almansi Theorem}
    Let $f:D\subset\mathbb{R}^n\to\mathbb{R}$ be a polyharmonic function, i.e. $\Delta^p_nf=0$ for some $p$, in a star-like domain $D$ with centre $0$. Then, there exist unique harmonic functions $h_0,\dots h_{p-1}$ in $D$ such that 
    \begin{equation}
    \label{eq classical almansi decomposition}
f(x)=h_0(x)+|x|^2h_1(x)+\dots |x|^{2p-2}h_{p-1}(x)=\sum_{j=0}^{p-1}|x|^{2j}h_j(x).
    \end{equation}
\end{thm}

\begin{oss}\label{oss:classicalAlmansi}
    As shown in \cite{PolyharmonicFunctions}, the components $h_j$ are constructed inductively, in the following way. Suppose $g_0,\dots, g_{p-2}$ are the Almansi components of $\Delta_n f$, then for $j=1,\dots, p-1$
    \begin{equation}
\label{eq componenti generiche}
    h_j(x)=\frac{1}{4j}\int_0^1\xi^{j-2+\frac{n}{2}}g_{j-1}(\xi x)d\xi,\qquad,\text{ for }j=1,\dots,p-2
\end{equation}
and 
    \begin{equation*}
        h_0(x)=f(x)-\sum_{j=1}^{p-1}r^{2j}h_j(x),
    \end{equation*}
    where $r(x)=\sqrt{x_1^2+\dots+x_n^2}$.
\end{oss}

\begin{cor}
    Let $\Omega_D\subset\mathbb{R}^{m+1}$ and let $f$ be a slice function or a circular slice function. Then, the components of the Almansi decomposition are slice functions or circular slice functions, too.
\end{cor}

\begin{ex}
\label{example classical almansi x4}
    Let us find the Almansi decomposition of the slice regular function $f:\mathbb{R}^6\to\mathbb{R}_5$, defined by $$f(x)=x^4=\alpha^4-6\alpha^2\beta^2+\beta^4+J\beta(4\alpha^3-4\alpha\beta^2),$$
    for any $x=x_0+\sum_{\ell=1}^5x_\ell e_\ell$, with $\alpha=x_0$, $\beta=\sqrt{\sum_{\ell=1}^5x_\ell^2}$ and $J=(\sum_{\ell=1}^5x_\ell e_\ell)/\beta$. By Theorem \ref{theorem laplacian any order slice regular functions}, $f\in\ker\Delta^3_6$ and so $\Delta^2_6 f$ is harmonic, which means that $\Delta^2_6 f$ has trivial Almansi decomposition. Recalling that $f'_s(x)=4\alpha^3-4\alpha\beta^2$, by applying \eqref{eq formula laplaciano iterato 1} with $k=0,1$ we find 
    \begin{align*}
    \Delta_6f(x)&=-8\frac{\partial}{\partial x}(4\alpha^3-4\alpha\beta^2)=-8\mathcal{I}\left(\frac{\partial}{\partial z}(4\alpha^3-4\alpha\beta^2)\right)=-16(3\alpha^2-\beta^2+2\alpha\beta J),\\
        \Delta_6f(x)&=-16\frac{\partial}{\partial x}\left(\beta^{-1}\partial_\beta(4\alpha^3-4\alpha\beta^2)\right)=-16\mathcal{I}\left(\frac{\partial}{\partial z}(-8\alpha)\right)=64.
    \end{align*}
    Following Remark \ref{oss:classicalAlmansi} we find the Almansi decomposition for $\Delta_6 f=h_0(x)+|x|^2h_1(x)\in \ker\Delta_6^2$:
    \begin{align*}
        h_1(x)&=\frac{1}{4}\int_0^1\xi^{1-2+3}\Delta_6^2f(\xi x)d\xi=\frac{16}{3}\\
h_0(x)&=\Delta_6f(x)-(\alpha^2+\beta^2)
h_1(x)=-\frac{160}{3}\alpha^2+\frac{32}{3}\beta^2-32\alpha\beta J.
    \end{align*}
    Now, let us set $g_0\coloneqq h_0$ and $g_1\coloneqq h_1$ and let us find the Almansi decomposition of $f(x)=h_0(x)+|x|^2h_1(x)+|x|^4h_2(x)$:
    \begin{equation*}
h_2(x)=\frac{1}{8}\int_0^1\xi^2-2+3g_1(\xi x)d\xi=\frac{1}{6},
    \end{equation*}
    \begin{equation*}
h_1(x)=\frac{1}{4}\int_0^1\xi^{1-2+3}g_0(\xi x)d\xi=-\frac{8}{3}\alpha^2+\frac{8}{15}\beta^2-\frac{8}{5}\alpha\beta J
    \end{equation*}
    and finally,
    \begin{equation*}
\begin{split}
    h_0(x)&=f(x)-|x|^2h_1(x)-|x|^4h_2(x)=\frac{7}{2}\alpha^4-\frac{21}{5}\alpha^2\beta^2
+\frac{3}{10}\beta^4+\beta J\left(\frac{28}{5}\alpha^3-\frac{12}{5}\alpha\beta^2\right).      \end{split}
    \end{equation*}
    Note that $h_0,h_1,h_2$ are slice functions.
\end{ex}
\begin{oss}
    The components of the Almansi decomposition of a polynomial can be obtained also through the so called Gauss or canonical decomposition. Indeed, the components of a homogeneous polynomial $p_n$ of degree $n$ are given by \cite[\S 2.1]{avery2017hyperspherical}
    \begin{equation*}
h_k(x)=\frac{(m+2n-4k-1)!!}{(2k)!!(m+2n-2k-1)!!}\sum_{j=0}^{\lfloor \frac{n}{2}-k\rfloor}\frac{(-1)^j(m+2n-4k-2j-3)!!}{(2j)!!(m+2n-2k-3)!!}|x|^{2j}\Delta_{m+1}^{j+k}(p_n),
    \end{equation*}
    with $k=1,\dots, \frac{m+1}{2}$.
\end{oss}

\subsection{Slice-Almansi decomposition in several variables}
The classical Almansi decomposition allows to reduce a polyharmonic function into a combination of harmonic functions. A similar machinery is provided in the context of quaternionic slice analysis \cite{Almansi}: slice regular functions (which are biharmonic) are combination of spherical derivatives (which are harmonic). These decomposition are called Almansi-type. Similar decompositions are then given for slice functions in one Clifford variable \cite{AlmansiPerottiClifford} and for several quaternionic variables \cite{almansiseveralquaternion}.
We now extend the Almansi-type decompositions to slice functions in several Clifford variables. We assume $\Omega_D\subset(\mathbb{R}^{m+1})^n$ open symmetric set.
\begin{defn}
\label{Definizione S}
Let $f\in\mathcal{S}(\Omega_D)$ and $H\in\mathcal{P}(n)$. For every $K=\{k_1,...,k_p\}\subset H$, with $k_1<...<k_p$, define over $\Omega_{D_H}$ the slice functions
\begin{equation*}
    \mathcal{S}^H_K(f):=\left(x_K\odot f\right)'_{s,H}=\left(\prod_{i=1}^px_{k_i}\odot f\right)'_{s,H},
\end{equation*}
and set $\mathcal{S}^\emptyset_\emptyset(f):=f$. If $H=\llbracket m\rrbracket:=\{1,2,...,m\}$ is an integers interval from 1 to some $m\in\{1,...,n\}$, we can write $\forall K\in\mathcal{P}(m)$
\begin{equation*}
    \mathcal{S}^{\llbracket m\rrbracket}_K(f):=(x_m^{\chi_K(m)}\dots(x_1^{\chi_K(1)}f)'_{s,1}\dots)'_{s,m},
    \end{equation*}
    where $\chi_K$ is the characteristic function of the set $K$. 
Note that, in this case, we can use the ordinary pointwise product as well as the slice product \cite[Proposition 2.52]{Several}.
If $f=\mathcal{I}(F)$, every $\mathcal{S}^H_K(f)$ is induced by the stem function
\begin{equation*}
    G^H_K(F):=\left(Z_K\otimes F\right)'_H,
\end{equation*}
where $Z_j\in Stem(\mathbb{C}^n)$ is the stem function $Z_j(\alpha_1+i\beta_1,...,\alpha_n+i\beta_n):=\alpha_j+e_j\beta_j$, inducing the monomial $x_j$, for any $j=1,...,n$.
\end{defn}


\begin{thm}
\label{Teorema principale}
Let $f\in\mathcal{S}(\Omega_D)$ and fix any $H\in\mathcal{P}(n)$, then
\begin{enumerate}
    \item we can decompose $f$ as
    \begin{equation}
    \label{Formula decomposizione di Almansi}
f(x)=\displaystyle\sum_{K\subset H}(-1)^{|H\setminus K|}\left(\overline{x}\right)_{H\setminus K}\odot\mathcal{S}^H_K(f)(x),
    \end{equation}
    where $\left(\overline{x}\right)_T=\prod_{j=1}^s\overline{x}_{t_j}$, if $T=\{t_1,\dots,t_s\}$, with $1\leq t_1<\dots<t_s\leq n$. If $H=\llbracket m\rrbracket$, \eqref{Formula decomposizione di Almansi} becomes
    \begin{equation}
    \label{Formula decomposizione di Almansi ordinata}
f(x)=\displaystyle\sum_{K\in\mathcal{P}(m)}(-1)^{|K^c|}\left(\overline{x}\right)_{K^c}\mathcal{S}^{\llbracket m\rrbracket}_K(f)(x),
    \end{equation}
    with $K^c=\{1,\dots,m\}\setminus K$.
    \item if $H\neq\llbracket n\rrbracket$,  $\mathcal{S}^H_K(f)\in\mathcal{S}_{c,H}(\Omega_{D_H})\cap\mathcal{S}_p(\Omega_{D_H})$, for every $ K\subset H$, where $p=\min H^c$, while $\mathcal{S}^{\llbracket n\rrbracket}_K\in\mathcal{S}_{c,\llbracket n\rrbracket}(\Omega_{D_{\llbracket n\rrbracket}})$, for every $ K\in\mathcal{P}(n)$;
    \item suppose $f\in\mathcal{S}\mathcal{R}(\Omega_D)$, then  $\Delta^{\gamma_m}_{m+1,h}\mathcal{S}^H_K(f)=0$, $\forall h\in H$, $\forall K\subset H$;

\item $f\in\mathcal{S}\mathcal{R}(\Omega_D)$ if and only if $\mathcal{S}^H_K(f)\in\mathcal{S}\mathcal{R}_p(\Omega_{D_H})$, $\forall H\in\mathcal{P}(n)\setminus\{1,\dots,n\}$, $K\subset H$, $p=\min H^c$;
\item $f\in\mathcal{S}_\mathbb{R}(\Omega_D)$ if, and only if, $\mathcal{S}^{\llbracket n\rrbracket}_K(f)$ is real valued, $\forall K\in\mathcal{P}(n)$.
\end{enumerate}

\end{thm}

\begin{proof}[Proof of Theorem \ref{Teorema principale}]
All the claims of the Theorem, but point $3$ follows the proof of \cite[Theorem 3.1]{almansiseveralquaternion}. Let us prove point $3$.
 Write $\mathcal{S}^H_K(f)=\left(x_h^{\chi_K(h)}\odot\mathcal{S}^{H\setminus\{h\}}_{K\setminus\{h\}}(f)\right)'_{s,h}$. By hypothesis, $f\in\ker(\partial/\partial x_h^c)$, then, by Proposition \ref{proposizione proprieta derivata sferica} (6) $\mathcal{S}^{H\setminus\{h\}}_{K\setminus\{h\}}(f)\in\ker(\partial/\partial x_h^c)$ and thanks to Leibniz formula \cite[Proposition 3.25]{Several}, $x_h^{\chi_K(h)}\odot\mathcal{S}^{H\setminus\{h\}}_{K\setminus\{h\}}(f)\in\ker(\partial/\partial x_h^c)$. Finally, by Proposition \ref{Prop:iteratedLaplaciansevvar}, $$\Delta_{m+1,h}^{\gamma_m}\mathcal{S}^H_K(f)=\Delta_{m+1,h}^{\gamma_m}\left(x_h^{\chi_K(h)}\odot\mathcal{S}^{H\setminus\{h\}}_{K\setminus\{h\}}(f)\right)'_{s,h}=0.$$

\end{proof}

\begin{oss}
    For any $H\in\mathcal{P}(n)$ we can define the linear operator
\begin{equation*}
    \mathcal{S}^H:\mathcal{S}(\Omega_D)\ni f\mapsto\left\{\mathcal{S}^H_K(f)\right\}_{K\subset H}\in\left(\mathcal{S}_{c,H}(\Omega_{D_H})\cap\mathcal{S}_p(\Omega_{D_H})\right)^{2^{|H|}},
\end{equation*}
where $p:=\min H^c$. whoose restriction to slice regular functions is
\begin{equation*}
    \mathcal{S}^H|_{\mathcal{SR}(\Omega_D)}:\mathcal{SR}(\Omega_D)\to\left(\mathcal{S}_{c,H}(\Omega_{D_H})\cap\mathcal{SR}_p(\Omega_{D_H}\cap\bigcap_{h\in H}\ker\Delta^{\gamma_m}_{m+1,h})\right)^{2^{|H|}},
\end{equation*}
\end{oss}
We give additional details about slice-Almansi decomposition, such as uniqueness of the components and a new one-variable characterization of slice regularity, with the following

\begin{prop}\label{prop:3claims}
    Let $f\in\mathcal{S}(\Omega_D)$ and fix $H\in\mathcal{P}(n)$, $h\in\{1,\dots,n\}$. 
    \begin{enumerate}
        \item Suppose there exist functions $\{h_K\}_{K\subset H}\subset\mathcal{S}_{c,H}(\Omega_D)$ such that
    \begin{equation*}
f(x)=\displaystyle\sum_{K\subset H}(-1)^{|H\setminus K|}\left(\overline{x}\right)_{H\setminus K}\odot h_K(x).
    \end{equation*}
Then $h_K=\mathcal{S}^H_K(f)$.
\item $f\in\mathcal{S}\mathcal{R}(\Omega_D)$ if and only if $\mathcal{S}^{\llbracket m\rrbracket}_K(f)\in\mathcal{S}\mathcal{R}_{m+1}(\Omega_D)$, $\forall m=0,...,n-1$, $K\in\mathcal{P}(m)$. 
\item Suppose $f\in\mathcal{S}_h(\Omega_D)$, then for every $K\in\mathcal{P}(h-1)$, $K\neq\{1,...,h-1\}$, it holds $\mathcal{S}^{\llbracket h\rrbracket}_K(f)=0$.
    In particular, \eqref{Formula decomposizione di Almansi ordinata} for $m=h$ reduces to
    \begin{equation*}
f=\sum_{K\in\mathcal{P}(h-1)}(-1)^{|K^c|}\overline{x}_{K^c}\mathcal{S}^{\llbracket h\rrbracket}_{K\cup\{h\}}(f)-\overline{x}_h\mathcal{S}^{\llbracket h\rrbracket}_{\llbracket h-1\rrbracket}(f).
    \end{equation*}
    \end{enumerate}
\end{prop}

\begin{proof}
    All the claims do not depend on the algebra in which the functions take values, so the proof is the same of the ones in the quaternionic setting, see \cite[Propositions 3.2, 3.4 and 4.3]{almansiseveralquaternion}
\end{proof}

Point 2 of Proposition \ref{prop:3claims} resembles the one-variable characterization of slice regularity given in Theorem \ref{thm one variable characterization}, in which iterations of spherical values and spherical derivatives (also referred as truncated spherical derivatives $\mathcal{D}_\epsilon(f)$) have been used.

\subsection{Simultaneous Almansi decompositions}

In this subsection we assume $\Omega_D\subset(\mathbb{R}^{m+1})^n$ a star-like w.r.t any variable, open and circular set.
\begin{cor}
\label{cor Almansi clifford several}
    Let $f\in\mathcal{SR}(\Omega_D)$ and for $H\in\mathcal{P}(n)$, let
    \begin{equation*}
f(x)=\displaystyle\sum_{K\subset H}(-1)^{|H\setminus K|}\left(\overline{x}\right)_{H\setminus K}\odot\mathcal{S}^H_K(f)(x)
    \end{equation*}
    be the Almansi-type decomposition of $f$ with respect to $H$. Then, for any $G\subset H$, we can further decompose 
    \begin{equation}
    \label{eq further Alamsi Rm}
f(x)=\sum_{K\subset H}\sum_{\substack{T\in\llbracket0,\frac{m-3}{2}\rrbracket^{|G|},\\T=(t_1,\dots,t_{|G|})}}(-1)^{|H\setminus K|}|x_G|^{2T}\left(\overline{x}\right)_{H\setminus K}\odot\mathcal{E}^{H,G}_{K,T}(f)(x),
    \end{equation}
    with $\mathcal{E}^{H,G}_{K,T}(f)(x)\in\ker\Delta_{m+1,G}$, where 
    \begin{equation*}
|x_G|^{2T}\coloneqq|x_{g_1}|^{2t_1}\cdot\dots\cdot|x_{g_s}|^{2t_s},
    \end{equation*}
    if $G=(g_1,\dots,g_s)$ and $T=(t_1,\dots,t_s)$. 
\end{cor}
\begin{proof}
    Let us prove \eqref{eq further Alamsi Rm} by induction over $|G|$. Suppose first that $G=\{g\}\subset H$, then, since $\mathcal{S}^H_K(f)\in\ker\Delta^{{\gamma_m}}_{m+1,H}$, for any $K\subset H$ and since $\Omega_D$ is a star-like domain with respect to $x_g$, by classical Almansi decomposition (Theorem \ref{Classical Almansi Theorem}) there exist $\mathcal{E}^{H,\{g\}}_{K,0},\dots,\mathcal{E}^{H,\{g\}}_{K,\frac{m-3}{2}}\in\ker\Delta_{m+1,g}$ such that, for any $K\subset H$,
    \begin{equation*}
\mathcal{S}^H_K(f)(x)=\sum_{j=0}^\frac{m-3}{2}|x_g|^{2j}\mathcal{E}^{H,\{g\}}_{K,j}(x)
    \end{equation*}
    and so
    \begin{equation*}
f(x)=\displaystyle\sum_{K\subset H}(-1)^{|H\setminus K|}\left(\overline{x}\right)_{H\setminus K}\odot\mathcal{S}^H_K(f)(x)=\displaystyle\sum_{K\subset H}\sum_{j=0}^\frac{m-3}{2}(-1)^{|H\setminus K|}|x_g|^{2j}\left(\overline{x}\right)_{H\setminus K}\odot\mathcal{E}^{H,\{g\}}_{K,j}(x).
    \end{equation*}
    Now, suppose that \eqref{eq further Alamsi Rm} holds for some $G\subset H$ and let us prove it for $\tilde G=G\cup\{g\}$, for some $g\in H\setminus G$. By induction, we have that 
    \begin{equation*}
f(x)=\sum_{K\subset H}\sum_{\substack{T\in\llbracket0,\frac{m-3}{2}\rrbracket^{|G|},\\T=(t_1,\dots,t_{|G|})}}(-1)^{|H\setminus K|}|x_G|^{2T}\overline{x}_{H\setminus K}\odot\mathcal{E}^{H,G}_{K,T}(f)(x),
    \end{equation*}
    with $\mathcal{E}^{H,G}_{K,T}(f)\in\ker\Delta_{m+1,g}^{\gamma_m}$, for every $ K\subset H$ and $T\subset G$. Thus, by Theorem \ref{Classical Almansi Theorem}, for any $K\subset H$, $T\subset G$ there exist $\{\mathcal{E}^{H,G\cup\{g\}}_{K,T\cup t_g}(f)\}_{t_g=0}^\frac{m-3}{2}\in\ker\Delta_{m+1,g}$ such that
    \begin{equation*}
       \mathcal{E}^{H,G}_{K,T}(f)(x)=\sum_{t_g=0}^\frac{m-3}{2}|x_g|^{2t_g}\mathcal{E}^{H,G\cup\{g\}}_{K,T\cup t_g}(f)(x) 
    \end{equation*}
    and so
    \begin{equation*}
\begin{split}
    f(x)&=\sum_{K\subset H}\sum_{\substack{T\in\llbracket0,\frac{m-3}{2}\rrbracket^{|G|},\\T=(t_1,\dots,t_{|G|})}}(-1)^{|H\setminus K|}|x_G|^{2T}\overline{x}_{H\setminus K}\odot\sum_{t_g=0}^\frac{m-3}{2}|x_g|^{2t_g}\mathcal{E}^{H,G\cup\{g\}}_{K,T\cup t_g}(f)(x) \\
    &=\sum_{K\subset H}\sum_{\substack{\tilde T\in\llbracket0,\frac{m-3}{2}\rrbracket^{|G|+1},\\\tilde T=T\cup t_g}}(-1)^{|H\setminus K|}|x_{G\cup\{g\}}|^{2\tilde T}\overline{x}_{H\setminus K}\odot\mathcal{E}^{H,G\cup\{g\}}_{K,\tilde T}(f)(x)\\
    &\sum_{K\subset H}\sum_{\substack{\tilde T\in\llbracket0,\frac{m-3}{2}\rrbracket^{|G|+1},\\\tilde T=(t_1,\dots,t_{|G|+1})}}(-1)^{|H\setminus K|}|x_{G\cup\{g\}}|^{2\tilde T}\overline{x}_{H\setminus K}\odot\mathcal{E}^{H,G\cup\{g\}}_{K,\tilde T}(f)(x).
\end{split}
    \end{equation*}
    
\end{proof}

\begin{ex}
\label{ex 2 variables}
    Let $f:(\mathbb{R}^{6})^2\to\mathbb{R}_5$, $f(x_1,x_2)=x_1^4x_2^7$. Then, choosing $H=\{1,2\}$, we can decompose $f$ as 
    \begin{equation*}
\begin{split}
    f(x_1,x_2)&=\sum_{K\subset\{1,2\}}(-1)^{\{1,2\}\setminus K}\overline{x}_{\{1,2\}\setminus K}\mathcal{S}^{\{1,2\}}_K(f)(x_1,x_2)\\
    &=\mathcal{S}^{\{1,2\}}_{1,2}-\overline{x}_1\mathcal{S}^{\{1,2\}}_2-\overline{x}_2\mathcal{S}^{\{1,2\}}_1+\overline{x}_1\overline{x}_2\mathcal{S}^{\{1,2\}}_\emptyset\\
    &=(x_1^5)'_{s,1}(x_2^8)'_{s,2}-\overline{x}_1(x_1^4)'_{s,1}(x_2^8)'_{s,2}-\overline{x}_2(x_1^5)'_{s,1}(x_2^7)'_{s,2}+\overline{x}_1\overline{x}_2(x_1^4)'_{s,1}(x_2^7)'_{s,2}\\
    &=(5\alpha_1^4-10\alpha_1^2\beta_1^2+\beta_1^4)(8\alpha_2^7-56\alpha_2^5\beta_2^2+56\alpha_2^3\beta_2^4-8\alpha_2\beta_2^6)+\\
    &\ -(\alpha_1-J_1\beta_1)(4\alpha_1^3-4\alpha_1\beta_1^2)(8\alpha_2^7-56\alpha_2^5\beta_2^2+56\alpha_2^3\beta_2^4-8\alpha_2\beta_2^6)+\\
    &\ -(\alpha_2-J_2\beta_2)(5\alpha_1^4-10\alpha_1^2\beta_1^2+\beta_1^4)(7\alpha_2^6-35\alpha_2^4\beta_2^2+21\alpha_2^2\beta_2^4-\beta_2^6)+\\
    &\ +(\alpha_1-J_1\beta_1)(\alpha_2-J_2\beta_2)(4\alpha_1^3-4\alpha_1\beta_1^2)(7\alpha_2^6-35\alpha_2^4\beta_2^2+21\alpha_2^2\beta_2^4-\beta_2^6).
\end{split}    
\end{equation*}
Note that by Proposition \ref{Prop potenza laplaciano derivata sferica}, for any $K\subset\{1,2\}$, $\mathcal{S}^{\{1,2\}}_K(f)\in\ker\Delta_{6,2}^2$ (and also in $\ker\Delta_{6,1}^2$), hence we can further decompose $\mathcal{S}^{\{1,2\}}_K(f)$ into harmonic functions with respect to $x_2$, through Classical Almansi decomposition: 
\begin{equation*}
    \mathcal{S}^{\{1,2\}}_K(f)=\mathcal{E}^{\{1,2\}}_{K,0}(f)+|x_2|^2\mathcal{E}^{\{1,2\}}_{K,2}(f),
\end{equation*}
with $\mathcal{E}^{\{1,2\}}_{K,T}(f)\in\ker\Delta_{6,2}$. This corresponds to the choice $G=\{2\}$ in Corollary \ref{cor Almansi clifford several}. By Remark \ref{oss:classicalAlmansi}, we find
\begin{equation*}
    \mathcal{E}^{\{1,2\}}_{K,1}(f)=\frac{1}{4}\int_0^1\xi^{1-2+3}\Delta_{6,2}(\mathcal{S}^{\{1,2\}}_K(f))(\xi x)d\xi,\quad\mathcal{E}^{\{1,2\}}_{K,0}(f)=\mathcal{S}^{\{1,2\}}_K(f)-|x_2|^2\mathcal{E}^{\{1,2\}}_{K,2}(f)
\end{equation*}
Let us start decomposing $\mathcal{S}^{\{1,2\}}_\emptyset(f)=((x_1^4)'_{s,1}x_2^7)'_{s,2}$.
First we need to compute $\Delta_{6,2}(\mathcal{S}^{\{1,2\}}_\emptyset(f))$ through \eqref{eq:Laplacianiderivatasfericasevvar1} applied to $(x_1^4)'_{s,1}x_2^7$ with $h=2$: observe that if $F=F_0+e_1F_1+e_2F_2+e_{12}F_{12}$ is the stem function inducing $f$, then $(x_1^4)'_{s,1}x_2^7=\mathcal{I}(G_0+e_2G_2)$, with $G_0=\beta_1^{-1}F_0$ and $G_2=\beta_1^{-1}F_{12}$, hence
\begin{align*}
    \Delta_{6,2}&(\mathcal{S}^{\{1,2\}}_\emptyset(f))(x)=2\beta_2^{-1}\sum_{K=0,1}[J_K,\partial_{\beta_2}(\beta_2^{-1}G_{K\cup\{2\}})]=2\beta_2^{-1}\partial_{\beta_2}(\beta_2^{-1}G_{2})=2\beta_2^{-1}\partial_{\beta_2}(\beta_2^{-1}\beta_1^{-1}F_{12})\\
&=2(x_1^4)'_{s,1}\beta_2^{-1}\partial_{\beta_2}(7\alpha_2^6-35\alpha_2^4\beta_2^2+21\alpha_2^2\beta_2^4-\beta_2^6)=2(x_1^4)'_{s,1}(-70\alpha_2^4+84\alpha_2^2\beta_2^2-6\beta_2^4)
\end{align*}
Thus
\begin{align*}
    \mathcal{E}^{\{1,2\}}_{\emptyset,1}(f)&=(x_1^4)'_{s,1}\left(12\alpha_2^6-36\alpha_2^4\beta_2^2+\frac{108}{7}\alpha_2^2\beta_2^4-\frac{4}{7}\beta_2^6\right),\\
    \mathcal{E}^{\{1,2\}}_{\emptyset,0}(f)&=(x_1^4)'_{s,1}\left(-5\alpha_2^4+6\alpha_2^2\beta_2^2-\frac{3}{7}\beta_2^4\right).
\end{align*}
In the very same way we decompose $\mathcal{S}^{\{1,2\}}_K(f)$, for $K=\{1\},\{2\},\{1,2\}$ by applying \eqref{eq:Laplacianiderivatasfericasevvar1} to $(x_1^5)'_{s,1}x_2^7$, $(x_1^4)'_{s,1}x_2^8$ and $(x_1^5)'_{s,1}x_2^8$ respectively, to get
\begin{equation*}
\begin{split}
\mathcal{S}^{\{1,2\}}_{1}(f)&=\mathcal{E}^{\{1,2\}}_{\{1\},0}(f)+|x_2|^2\mathcal{E}^{\{1,2\}}_{\{1\},1}(f)\\
&=(x_1^5)'_{s,1}\left(12\alpha_2^6-36\alpha_2^4\beta_2^2+\frac{108}{7}\alpha_2^2\beta_2^4-\frac{4}{7}\beta_2^6\right)+|x_2|^2(x_1^5)'_{s,1}\left(-5\alpha_2^4+6\alpha_2^2\beta_2^2-\frac{3}{7}\beta_2^4\right)
\end{split}
\end{equation*}
\begin{equation*}
\begin{split}
\mathcal{S}^{\{1,2\}}_{2}(f)&=\mathcal{E}^{\{1,2\}}_{\{2\},0}(f)+|x_2|^2\mathcal{E}^{\{1,2\}}_{\{2\},1}(f)\\
&=(x_1^4)'_{s,1}(15\alpha_2^7-63\alpha_2^5\beta_2^2+45\alpha_2^3\beta_2^4-5\alpha_2\beta_2^6)+|x_2|^2(x_1^4)'_{s,1}(-7\alpha_2^5+14\alpha_2^3\beta_2^2-3\alpha_2\beta_2^4)
\end{split}
\end{equation*}
\begin{equation*}
\begin{split}
\mathcal{S}^{\{1,2\}}_{1,2}(f)&=\mathcal{E}^{\{1,2\}}_{\{1,2\},0}(f)+|x_2|^2\mathcal{E}^{\{1,2\}}_{\{1,2\},1}(f)\\
&=(x_1^5)'_{s,1}(15\alpha_2^7-63\alpha_2^5\beta_2^2+45\alpha_2^3\beta_2^4-5\alpha_2\beta_2^6)+|x_2|^2(x_1^5)'_{s,1}(-7\alpha_2^5+14\alpha_2^3\beta_2^2-3\alpha_2\beta_2^4)
\end{split}
\end{equation*}
Thus, we can further decompose $f$ as
\begin{equation*}
    \begin{split}
f&=\mathcal{E}^{\{1,2\}}_{\{1,2\},0}(f)+|x_2|^2\mathcal{E}^{\{1,2\}}_{\{1,2\},1}(f)-\overline{x}_1\mathcal{E}^{\{1,2\}}_{\{2\},0}(f)-\overline{x}_1|x_2|^2\mathcal{E}^{\{1,2\}}_{\{2\},1}(f)+\\
&-\overline{x}_2\mathcal{E}^{\{1,2\}}_{\{1\},0}(f)-\overline{x}_2|x_2|^2\mathcal{E}^{\{1,2\}}_{\{1\},1}(f)+\overline{x}_1\overline{x}_2\mathcal{E}^{\{1,2\}}_{\emptyset,0}(f)+\overline{x}_1\overline{x}_2|x_2|^2\mathcal{E}^{\{1,2\}}_{\emptyset,1}(f),
    \end{split}
\end{equation*}
where for any $K\subset\{1,2\}$, $T=0,1$ it holds $$\Delta^2_{6,1}\mathcal{E}^{\{1,2\}}_{K,T}=\Delta_{6,2}\mathcal{E}^{\{1,2\}}_{K,T}=0.$$ 
\end{ex}
\section{Fueter-Sce theorem in several variables}

Fueter-Sce theorem is a fundamental result in hypercomplex analysis, which constitutes a bridge between slice regular and monogenic functions. 
We take from \cite{FueterSceMapping} a modern version of it.

\begin{thm}[Fueter-Sce theorem]
    Let $m$ be odd and let $\Omega_D\subset\mathbb{R}^{m+1}$. Then $\Delta_{m+1}^{\gamma_m}f$ is monogenic, namely
    \begin{equation*}
\overline{\partial}\Delta_{m+1}^{\gamma_m}f=0.
    \end{equation*}
\end{thm}
\begin{oss}
    In \cite{qian}, the previous result was exteded to Clifford algebras with an even number of imaginary units, requiring techniques of fractional differential operators. 
\end{oss}

We need some preliminary results to extend Fueter-Sce Theorem in several variables. From now, assume $\Omega_D\subset(\mathbb{R}^{m+1})^n$ open and symmetric subset. Denote with $\mathcal{AM}(\Omega_D)\coloneqq\mathcal{S}(\Omega_D)\cap\mathcal{M}(\Omega_D)$ the set of axially monogenic functions, namely the set of functions which are both slice and monogenic. We also set $\mathcal{AM}_h(\Omega_D)\coloneqq\mathcal{S}_h(\Omega_D)\cap\mathcal{M}_h(\Omega_D)$, with $\mathcal{M}_h(\Omega_D)=\{f:\Omega_D\to\mathbb{R}_m:\overline{\partial}_{x_h}f=0\}$.
\begin{lem}
\label{lemma derivata sferica laplaciano e dirac}
Suppose $f\in\mathcal{S}\mathcal{R}_h(\Omega_D)$, for some $h=1,\dots, n$, then the following hold:
\begin{enumerate}
    \item $\overline{\partial}_{x_h}f=\frac{1-m}{2}f'_{s,h}$;
    \item $\Delta_{m+1,h} f=2(1-m)\dfrac{\partial f'_{s,h}}{\partial x_h}=2(1-m)\partial_{x_h}(f'_{s,h})$.
\end{enumerate}
\end{lem}
\begin{proof}
\begin{enumerate}
    \item Note that $\forall y=(y_1,...,y_n)\in\Omega_D$, $f^y_h\in\mathcal{S}\mathcal{R}(\Omega_{D,h}(y))$, then we can apply (\ref{equazione derivata sferica parziale coincide con unidimensionale}) and \cite[Proposition 9]{perotti2022cauchy} to get
\begin{equation*}
 \overline{\partial}_{x_h}f(y)=\overline{\partial}(f^y_h)(y_h)=\frac{1-m}{2}(f^y_h)'_s(y_h)=\frac{1-m}{2}f'_{s,h}(y).
\end{equation*}
\item By (\ref{equazione derivata sferica parziale coincide con unidimensionale}), \cite[Proposition 9]{perotti2022cauchy} and \cite[Theorem 2.2 (ii)]{Global} we have
\begin{equation*}
    \begin{split}
        \Delta_hf(y)=\Delta(f^y_h)(y_h)&=4\frac{1-m}{2}\dfrac{\partial (f^y_h)'_s}{\partial x}(y_h)=2(1-m)\theta(f^y_h)'_s(y_h)=2(1-m)\partial(f^y_h)'_s(y_h)\\
        &=2(1-m)\partial_{x_h}f'_{s,h}(y),
    \end{split}
\end{equation*}
where $(\theta f)(x)=\frac{1}{2}\left(\frac{\partial f}{\partial \alpha}(x)+\frac{\operatorname{Im}(x)}{|\operatorname{Im}(x)|^2}(\beta\frac{\partial f}{\partial\beta}(x)+\gamma\frac{\partial f}{\partial\gamma}(x)+\delta\frac{\partial f}{\partial\delta}(x))\right)$ satisfies $\theta f=\frac{\partial f}{\partial x}$ and $2\theta f'_s=\partial f'_s$  for any slice function $f$.
\end{enumerate}
\end{proof}

\begin{thm}[Futer-Sce theorem in several variables]
\label{thm fueter sce several variables}
    Let $f\in\mathcal{S}\mathcal{R}_h(\Omega_D)$, for some $h=1,\dots,n$. Then $\Delta_{m+1,h}^{\gamma_m}f\in\mathcal{AM}_h(\Omega_D)$, namely it holds
    \begin{equation*}
\overline{\partial}_{x_h}\Delta_{m+1,h}^{\gamma_m}f=0.
    \end{equation*}
    Thus, the Fueter-Sce map extends to
    \begin{equation*}
\Delta_{m+1,h}^{\gamma_m}:\mathcal{S}\mathcal{R}_h(\Omega_D)\to\mathcal{A}\mathcal{M}_h(\Omega_D).
    \end{equation*}
    \end{thm}

\begin{proof}
Since $f\in\mathcal{S}\mathcal{R}_h(\Omega_D)$, we can apply Lemma \ref{lemma derivata sferica laplaciano e dirac} 
and Proposition \ref{Prop:iteratedLaplaciansevvar} 
    \begin{equation*}
\overline{\partial}_{x_h}\Delta_{m+1,h}^{\gamma_m}f=\Delta_{m+1,h}^{\gamma_m}\overline{\partial}_{x_h}f=\frac{1-m}{2}\Delta_{m+1,h}^{\gamma_m}f'_{s,h}=0.
    \end{equation*}
\end{proof}
We can also use Almansi decomposition for several Clifford variables to find new relations with the theory of axially monogenic functions and give another proof of Fueter-Sce Theorem in several variables.

\begin{prop}
    Let $f\in\mathcal{S}\mathcal{R}(\Omega_D)$. Then for every $h=1,\dots n$, the components of the ordered Almansi decomposition of $f$, $\mathcal{S}^{\llbracket m\rrbracket}_K(f)$ can be written as
\begin{equation}
\label{equazione scrittura componenti ordinate con CRF Clifford}
    \mathcal{S}^{\llbracket m\rrbracket}_K(f)=\left(\frac{1-m}{2}\right)^m\overline{\partial}_{x_m}(x_m^{\chi_K(m)}\dots\overline{\partial}_{x_1}(x_1^{\chi_K(1)}f)\dots).
\end{equation}
\end{prop}
\begin{proof}
Recall that if $f\in\mathcal{S}\mathcal{R}(\Omega_D)$, by Proposition \ref{prop:3claims} (2), $f\in\mathcal{SR}_1(\Omega_D)$ and $f'_{s,\llbracket j\rrbracket}\in\mathcal{S}\mathcal{R}_{j+1}$, for any $j=1,\dots, n-1$. Then, we can iteratively apply (1) of Lemma \ref{lemma derivata sferica laplaciano e dirac} with $h=1,\dots,m$ to the definition of $\mathcal{S}^{\llbracket m\rrbracket}_K(f)$ to obtain \eqref{equazione scrittura componenti ordinate con CRF Clifford}.
\end{proof}

\begin{prop}
    Let $f\in\mathcal{S}\mathcal{R}(\Omega_D)$ and let $h=1,\dots,n-1$. Then, for every $K\in\mathcal{P}(h)$, $\mathcal{S}^{\llbracket m\rrbracket}_K(f)$ satisfies the following:
    \begin{enumerate}
\item $\partial_{x_h}\left(\Delta_{m+1,h}^\frac{m-3}{2}\mathcal{S}^{\llbracket h\rrbracket}_K(f)\right)\in\mathcal{A}\mathcal{M}_h(\Omega_D)$;
\item $\Delta_{m+1,h+1}^{\gamma_m}\mathcal{S}^{\llbracket h\rrbracket}_K(f)\in\mathcal{A}\mathcal{M}_{h+1}(\Omega_D)$.
    \end{enumerate}
\end{prop}
\begin{proof}
    \begin{enumerate}[wide, , labelindent=0pt]
\item It holds
\begin{equation*}
    \overline{\partial}_{x_h}\partial_{x_h}\Delta_{m+1,h}^\frac{m-3}{2}\mathcal{S}^{\llbracket h\rrbracket}_K(f)=\frac{1}{4}\Delta_{m+1,h}^{\gamma_m}\mathcal{S}^{\llbracket h\rrbracket}_K(f)=0,
\end{equation*}
by Theorem \ref{Teorema principale}.
\item Similarly,
\begin{equation*}
    \begin{split}
\overline{\partial}_{x_{h+1}}\Delta_{m+1,h+1}^{\gamma_m}\mathcal{S}^{\llbracket h\rrbracket}_K(f)&=\Delta_{m+1,h+1}^{\gamma_m}\overline{\partial}_{x_{h+1}}\mathcal{S}^{\llbracket h\rrbracket}_K(f)=\frac{1-m}{2}\Delta_{m+1,h+1}^{\gamma_m}\overline{\partial}_{x_{h+1}}\left(\mathcal{S}^{\llbracket h\rrbracket}_K(f)\right)'_{s,h+1}\\
&=\frac{1-m}{2}\Delta_{m+1,h+1}^{\gamma_m}\overline{\partial}_{x_{h+1}}\mathcal{S}^{\llbracket h+1\rrbracket}_K(f)=0,
    \end{split}
\end{equation*}
again by Theorem \ref{Teorema principale}, Lemma \ref{lemma derivata sferica laplaciano e dirac} and Proposition \ref{Prop:iteratedLaplaciansevvar}.
    \end{enumerate}
\end{proof}

\begin{lem}
    Let $f\in\mathcal{S}^1(\Omega_D)\cap\ker(\partial/\partial x_h^c)$, then it holds
    \begin{equation*}
\Delta_{m+1,h}f=2(1-m)\sum_{K\in\mathcal{P}(h-1)}(-1)^{|K^c|}\left(\overline{x}\right)_{K^c}\partial_{x_h}\left(\mathcal{S}^{\llbracket h\rrbracket}_K(f)\right),
    \end{equation*}
    with $K^c=\{1,\dots, h-1\}\setminus K$.
\end{lem}
    \begin{proof}
By Theorem \ref{Teorema principale}, we can decompose $f$ as
\begin{equation*}
    f=\sum_{K\in\mathcal{P}(h-1)}(-1)^{|K^c|}\left(\overline{x}\right)_{K^c}\mathcal{S}^{\llbracket h-1\rrbracket}_K(f),
\end{equation*}
then
\begin{equation*}
    \begin{split}
\Delta_{m+1,h}f&=\Delta_{m+1,h}\left(\sum_{K\in\mathcal{P}(h-1)}(-1)^{|K^c|}\left(\overline{x}\right)_{K^c}\mathcal{S}^{\llbracket h-1\rrbracket}_K(f)\right)\\
&=\sum_{K\in\mathcal{P}(h-1)}(-1)^{|K^c|}\left(\overline{x}\right)_{K^c}\Delta_{m+1,h}\left(\mathcal{S}^{\llbracket h-1\rrbracket}_K(f)\right).
    \end{split}
\end{equation*}
Now, by Lemma \ref{lemma derivata sferica laplaciano e dirac} (1) we have
\begin{equation*}
    \Delta_{m+1,h}\left(\mathcal{S}^{\llbracket h-1\rrbracket}_K(f)\right)=4\partial_{x_h}\overline{\partial}_{x_h}\left(\mathcal{S}^{\llbracket h-1\rrbracket}_K(f)\right)=2(1-m)\partial_{x_h}\left(\mathcal{S}^{\llbracket h\rrbracket}_K(f)\right)
\end{equation*}
and so
\begin{equation*}
    \Delta_{m+1,h}f=2(1-m)\sum_{K\in\mathcal{P}(h-1)}(-1)^{|K^c|}\left(\overline{x}\right)_{K^c}\partial_{x_h}\left(\mathcal{S}^{\llbracket h\rrbracket}_K(f)\right).
\end{equation*}
    \end{proof}

As promised, we can give a new proof of Theorem \ref{thm fueter sce several variables}, through Slice-Almansi decomposition.
    \begin{proof}[Proof of Theorem \ref{thm fueter sce several variables}]
    By Lemma \ref{lemma derivata sferica laplaciano e dirac} (2) it holds
    \begin{equation*}
\Delta_{m+1,h}f=2(1-m)\sum_{K\in\mathcal{P}(h-1)}(-1)^{|K^c|}\left(\overline{x}\right)_{K^c}\partial_{x_h}\left(\mathcal{S}^{\llbracket h\rrbracket}_K(f)\right)
    \end{equation*}
and recall that, since $f\in\mathcal{S}\mathcal{R}_h(\Omega_D)$, $\mathcal{S}^{\llbracket h\rrbracket}_K(f)=0$, for every $K\in\mathcal{P}(h-1)\setminus\llbracket h-1\rrbracket$ (Propositions \ref{proposizione proprieta derivata sferica} (1) and \ref{proposizione proprieta derivata sferica} (6)). Thus, the previous equation reduces to
\begin{equation*}
    \Delta_{m+1,h}f=2(1-m)\partial_{x_h}\left(\mathcal{S}^{\llbracket h\rrbracket}_{\llbracket h-1\rrbracket}(f)\right),
\end{equation*}
so, again by Lemma \ref{lemma derivata sferica laplaciano e dirac} (2)
\begin{equation*}
    \Delta_{m+1,h}^{\gamma_m}f=\Delta_{m+1,h}^\frac{m-3}{2}\Delta_{m+1,h}^{\gamma_m}f=2(1-m)\partial_{x_h}\left(\Delta_{m+1,h}^\frac{m-3}{2}\mathcal{S}^{\llbracket h\rrbracket}_{\llbracket h-1\rrbracket}(f)\right)
\end{equation*}
and we conclude with Proposition \ref{Prop:iteratedLaplaciansevvar}.
    \end{proof}

    Finally, we extend \cite[Corollary 2]{AlmansiPerottiClifford} to several variables.
\begin{cor}
\label{cor further decomposition several variables}
    With the notation of Corollary \ref{cor Almansi clifford several}, let 
    \begin{equation*}
f(x)=\sum_{K\subset H}\sum_{\substack{T\in\llbracket0,\frac{m-3}{2}\rrbracket^{|G|},\\T=(t_1,\dots,t_{|G|})}}(-1)^{|H\setminus K|}|x_G|^{2T}\left(\overline{x}\right)_{H\setminus K}\odot\mathcal{E}^{H,G}_{K,T}(f)(x)
    \end{equation*}
    be the decomposition \eqref{eq further Alamsi Rm}, with harmonic components $\mathcal{E}^{H,G}_{K,T}(f)$. For any $T\in\llbracket0,\frac{m-3}{2}\rrbracket^{|G|}$, define 
    \begin{equation*}
\mathcal{G}^{H,G}_T(f)=\sum_{K\subset H}(-1)^{|H\setminus K|}\left(\overline{x}\right)_{H\setminus K} \odot\mathcal{E}^{H,G}_{K,T}(f)\in\mathcal{S}(\Omega_D).
    \end{equation*}
    Then we can write
    \begin{equation}
    \label{eq ricombined almansi decomposition several clifford}
f(x)=\sum_{\substack{T\in\llbracket0,\frac{m-3}{2}\rrbracket^{|G|},\\T=(t_1,\dots,t_{|G|})}}|x_G|^{2T}\mathcal{G}^{H,G}_T(f)(x),
    \end{equation}
    where $\mathcal{G}^{H,G}_T(f)$ are biharmonic, namely
$\Delta_{m+1,g}^2\mathcal{G}^{H,G}_T(f)=0$, for every $g\in G$.
\end{cor}

\begin{proof}
    Decomposition \eqref{eq ricombined almansi decomposition several clifford} follows by \eqref{eq further Alamsi Rm} and the definition of $\mathcal{G}^{H,G}_T$, now let us prove that the components are biharmonic in every variable $x_g$, with $g\in G$. Indeed, it holds
    \begin{equation*}
\begin{split}
    \Delta_{m+1,g}\mathcal{G}^{H,G}_T&=\sum_{K\subset H,g\notin K}(-1)^{|H\setminus K|}\left(\overline{x}\right)_{H\setminus K} \odot\Delta_{m+1,g}\left(\mathcal{E}^{H,G}_{K,T}(f)\right)+\\
    &+\sum_{K\subset H,g\notin K}(-1)^{|H\setminus (K\cup\{g\})|}\left(\overline{x}\right)_{H\setminus K} \odot\Delta_{m+1,g}\left(\overline{x}_g\odot\mathcal{E}^{H,G}_{K\cup\{g\},T}(f)\right)\\
    &=\sum_{K\subset H,g\notin K}(-1)^{|H\setminus (K\cup\{g\})|}\left(\overline{x}\right)_{H\setminus K} \odot\partial_{x_g}\mathcal{E}^{H,G}_{K\cup\{g\},T}(f),
\end{split}
    \end{equation*}
    where we have used that $\Delta_{m+1,g}\mathcal{E}^{H,G}_{K,T}=0$, for every $g\in G$ and that $\Delta_{m+1,g}\left(\overline{x}_g\odot\mathcal{E}^{H,G}_{K\cup\{g\},T}(f)\right)=\partial_{x_g}\mathcal{E}^{H,G}_{K\cup\{g\},T}(f)$. Finally
    \begin{equation*}
\begin{split}
    \Delta_{m+1,g}^2\mathcal{G}^{H,G}_T&=\Delta_{m+1,g}\left(\sum_{K\subset H,g\notin K}(-1)^{|H\setminus (K\cup\{g\})|}\left(\overline{x}\right)_{H\setminus K} \odot\partial_{x_g}\mathcal{E}^{H,G}_{K\cup\{g\},T}(f)\right)\\
    &=\sum_{K\subset H,g\notin K}(-1)^{|H\setminus (K\cup\{g\})|}\left(\overline{x}\right)_{H\setminus K} \odot\partial_{x_g}\Delta_{m+1,g}\left(\mathcal{E}^{H,G}_{K\cup\{g\},T}(f)\right)=0.
\end{split}
    \end{equation*}
\end{proof}

\begin{oss}
    We remark that decomposition \eqref{eq ricombined almansi decomposition several clifford} is formally equivalent to the Classical Almansi decomposition. By the uniqueness of Almansi decomposition, the functions $\mathcal{G}^{H,G}_T$ in general belong to the kernel of $\ker\Delta^2_{m+1,g}\setminus\ker\Delta_{m+1,g}$, for any $g\in G$. 
\end{oss}
\begin{oss}
    Note that in the one variable case (\cite[Corollary 2]{AlmansiPerottiClifford}) the components were more than biharmonic functions, namely they were in the kernel of the third-order differential operator $\overline{\partial}\Delta$. In several variables, the same happens for $x_1$, but in general it doesn't hold for the other variables.
\end{oss}

\begin{ex}
Let us resume Example \ref{ex 2 variables}, where we decomposed the function $f:(\mathbb{R}^6)^2\to\mathbb{R}_5$, $f(x_1,x_2)=x_1^4x_2^7$ as 
\begin{equation*}
    \begin{split}
f&=\mathcal{E}^{\{1,2\}}_{\{1,2\},0}(f)+|x_2|^2\mathcal{E}^{\{1,2\}}_{\{1,2\},1}(f)-\overline{x}_1\mathcal{E}^{\{1,2\}}_{\{2\},0}(f)-\overline{x}_1|x_2|^2\mathcal{E}^{\{1,2\}}_{\{2\},1}(f)+\\
&-\overline{x}_2\mathcal{E}^{\{1,2\}}_{\{1\},0}(f)-\overline{x}_2|x_2|^2\mathcal{E}^{\{1,2\}}_{\{1\},1}(f)+\overline{x}_1\overline{x}_2\mathcal{E}^{\{1,2\}}_{\emptyset,0}(f)+\overline{x}_1\overline{x}_2|x_2|^2\mathcal{E}^{\{1,2\}}_{\emptyset,1}(f),
    \end{split}
\end{equation*}
Following Corollary \ref{cor further decomposition several variables}, for every $T=0,1$, define
\begin{equation*}
    \mathcal{G}_T=\sum_{K\in\mathcal{P}(2)}(-1)^{|\{1,2\}\setminus K|}\left(\overline{x}\right)_{\{1,2\}\setminus K}\mathcal{E}^{\{1,2\}}_{K,T}(f).
\end{equation*}
Explicitly, we have
\begin{equation*}
    \begin{split}
\mathcal{G}_0&=\mathcal{E}^{\{1,2\}}_{\{1,2\},0}(f)-\overline{x}_1\mathcal{E}^{\{1,2\}}_{\{2\},0}(f)-\overline{x}_2\mathcal{E}^{\{1,2\}}_{\{1\},0}(f)+\overline{x}_1\overline{x}_2\mathcal{E}^{\{1,2\}}_{\emptyset,0}(f)\\
&=\left[(x_1^5)'_{s,1}-\overline{x}_1(x_1^4)'_{s,1}\right](15\alpha_2^7-63\alpha_2^5\beta_2^2+45\alpha_2^3\beta_2^4-5\alpha_2\beta_2^6)+\\
&-\left[(x_1^5)'_{s,1}-\overline{x}_1(x_1^4)'_{s,1}\right]\overline{x}_2\left(12\alpha_2^6-36\alpha_2^4\beta_2^2+\frac{108}{7}\alpha_2^2\beta_2^4-\frac{4}{7}\beta_2^6\right)\\
&=x_1^4\left[15\alpha_2^7-63\alpha_2^5\beta_2^2+45\alpha_2^3\beta_2^4-5\alpha_2\beta_2^6-\overline{x}_2\left(12\alpha_2^6-36\alpha_2^4\beta_2^2+\frac{108}{7}\alpha_2^2\beta_2^4-\frac{4}{7}\beta_2^6\right)\right],
    \end{split}
\end{equation*}
\begin{equation*}
    \begin{split}
    \mathcal{G}_1&=\mathcal{E}^{\{1,2\}}_{\{1,2\},1}(f)-\overline{x}_1\mathcal{E}^{\{1,2\}}_{\{2\},1}(f)-\overline{x}_2\mathcal{E}^{\{1,2\}}_{\{1\},1}(f)+\overline{x}_1\overline{x}_2\mathcal{E}^{\{1,2\}}_{\emptyset,1}(f)\\
    &=\left[(x_1^5)'_{s,1}-\overline{x}_1(x_1^4)'_{s,1}\right](-7\alpha_2^5+14\alpha_2^3\beta_2^2-3\alpha_2\beta_2^4)+\\
&-\left[(x_1^5)'_{s,1}-\overline{x}_1(x_1^4)'_{s,1}\right]\overline{x}_2\left(-5\alpha_2^4+6\alpha_2^2\beta_2^2-\frac{3}{7}\beta_2^4\right)\\
&=x_1^4\left[-7\alpha_2^5+14\alpha_2^3\beta_2^2-3\alpha_2\beta_2^4-\overline{x}_2\left(-5\alpha_2^4+6\alpha_2^2\beta_2^2-\frac{3}{7}\beta_2^4\right)\right].
    \end{split}
\end{equation*}

With these functions, we decompose
\begin{equation*}
    f(x_1,x_2)=\mathcal{G}_0(x_1,x_2)+|x_2|^2\mathcal{G}_{1}(x_1,x_2),
\end{equation*}
which corresponds to decomposition \eqref{eq ricombined almansi decomposition several clifford}.
Note that $\mathcal{G}_{0},\mathcal{G}_{1}\in\ker\Delta_{6,j}^2$, for $j=1,2$.

\end{ex}

\textbf{Conflicts of interest} The author declare no Conflict of interest.

\printbibliography
\end{document}